\documentclass{amsart}

\usepackage{amsmath}
\usepackage{bm} 
\usepackage{verbatim} 
\usepackage{graphicx} 
\usepackage{multirow} 
\usepackage{tikz} 
\usepackage{tikz-cd} 

\newcommand{\tikzbox}[1]{%
  \begin{tikzpicture}%
    \node[draw,shape=rectangle] at (0,0) {$#1$};%
  \end{tikzpicture}%
}

\usetikzlibrary{shapes}
\pgfqkeys{/tikz/commutative diagrams}{%
  path/.style={/tikz/arrows=cm cap-cm cap}%
}

\makeatletter \define@key{meshkeys}{u}{\def\myu{#1}}
\define@key{meshkeys}{v}{\def\myv{#1}}
\define@key{meshkeys}{r}{\def\myr{#1}}
\tikzdeclarecoordinatesystem{mesh}%
{%
  \setkeys{meshkeys}{#1}%
  \pgfpointadd{\pgfpointxy{\myr*(\myu+\myv)*cos(45)}%
    {\myr*(\myv-\myu)*sin(45)}}%
}

\makeatletter \define@key{meshkeys}{u}{\def\myu{#1}}
\define@key{meshkeys}{v}{\def\myv{#1}}
\define@key{meshkeys}{r}{\def\myr{#1}}
\tikzdeclarecoordinatesystem{smesh}%
{%
  \setkeys{meshkeys}{#1}%
  \pgfpointadd{\pgfpointxy{\myr*(\myu+\myv)*cos(35)}%
    {\myr*(\myv-\myu)*sin(35)}}%
}

\makeatletter \define@key{meshkeys}{u}{\def\myu{#1}}
\define@key{meshkeys}{v}{\def\myv{#1}}
\define@key{meshkeys}{r}{\def\myr{#1}}
\tikzdeclarecoordinatesystem{eqmesh}%
{%
  \setkeys{meshkeys}{#1}%
  \pgfpointadd{\pgfpointxy{\myr*(\myu+\myv)*cos(60)}%
    {\myr*(\myv-\myu)*sin(60)}}%
}

\DeclareMathOperator{\add}{add} 
 \DeclareMathOperator{\End}{End}
\DeclareMathOperator{\Ext}{Ext} \DeclareMathOperator{\Hom}{Hom}

\DeclareMathOperator{\coh}{coh}
\DeclareMathOperator{\vect}{vect} 
\DeclareMathOperator{\qgr}{qgr} \DeclareMathOperator{\gr}{gr}
\DeclareMathOperator{\slope}{slope} \DeclareMathOperator{\rk}{rk}
\DeclareMathOperator{\lcm}{\mathrm{lcm}}

\newcommand{\cf}{\emph{cf.}~} \newcommand{\ie}{\emph{i.e.}~}

\newcommand{\NN}{\mathbb{N}} 
\newcommand{\ZZ}{\mathbb{Z}}
\newcommand{\QQ}{\mathbb{Q}}

\newcommand{\xto}{\xrightarrow}

\newcommand{\set}[1]{\left\{#1\right\}}
\newcommand{\setP}[2]{\set{#1\mid#2}}

\newcommand{\cat}[1]{\mathcal{#1}} 
\newcommand{\C}{\cat{C}}
\newcommand{\M}{\mathcal{M}}

\newcommand{\T}{\mathcal{T}}
\newcommand{\K}{\mathrm{K}}
\renewcommand{\mod}{\mathrm{mod}\,} \newcommand{\op}{\mathrm{op}}

\newcommand{\Db}{\mathrm{D}^\mathrm{b}}

\newcommand{\OO}{\cat{O}} 
\newcommand{\HH}{\cat{H}} 
\newcommand{\A}{\cat{H}} 
\newcommand{\LL}{\mathbb{L}}
\renewcommand{\P}{\mathbb{P}} 
\newcommand{\XX}{\mathbb{X}}
\newcommand{\YY}{\mathbb{Y}}
\newcommand{\cohX}{\coh\XX} 
\newcommand{\vecx}{\vec{x}}
\newcommand{\vecy}{\vec{y}} 
\newcommand{\vecz}{\vec{z}}
\newcommand{\vecc}{\vec{c}}
\newcommand{\vecw}{\vec{\omega}}
\newcommand{\p}{\mathbf{p}} 
\newcommand{\blambda}{\bm{\lambda}}

\DeclareMathOperator{\Jac}{Jac} 

\newcommand{\Q}{\tilde{Q}}

\numberwithin{figure}{section} \numberwithin{table}{section}

\newtheorem{theorem}{Theorem}[section]
\newtheorem{proposition}[theorem]{Proposition}
\newtheorem{lemma}[theorem]{Lemma}

\theoremstyle{definition} \newtheorem{definition}[theorem]{Definition}

\newtheorem*{acknowledgements}{Acknowledgements}

\theoremstyle{remark} 

\begin{document}

\title{$\tau^2$-stable tilting complexes over weighted projective
  lines}

\author[G. Jasso]{Gustavo Jasso} \address{Mathematisches Institut, Universit\"at Bonn, Endenicher Allee 60, 53115 Bonn, Germany}
\email{gjasso@math.uni-bonn.de}

\begin{abstract}
  Let $\XX$ be a weighted projective line and $\cohX$ the associated
  categoy of coherent sheaves.  We classify the tilting complexes $T$
  in $\Db(\cohX)$ such that $\tau^2T\cong T$, where $\tau$ is the
  Auslander-Reiten translation in $\Db(\cohX)$.  As an application of
  this result, we classify the 2-representation-finite algebras which
  are derived-equivalent to a canonical algebra.  This complements
  Iyama-Oppermann's classification of the iterated tilted
  2-representation-finite algebras.  By passing to 3-preprojective
  algebras, we obtain a classification of the selfinjective
  cluster-tilted algebras of canonical-type.  This complements
  Ringel's classification of the selfinjective cluster-tilted
  algebras.
\end{abstract}

\maketitle

\section{Introduction}

Let $\XX$ be a weighted projective line over an algebraically closed
field and
\[
\tau\colon \Db(\cohX)\to\Db(\cohX)
\]
be the Auslander-Reiten translation in the bounded derived category of
$\cohX$, see \cite{geigle_class_1985} for definitions.  The following
objects, which are closely related to each other, are classified in
this article:
\begin{enumerate}
\item \label{intro:tau2-stable} The $\tau^2$-stable tilting complexes
  in $\Db(\cohX)$,
\item \label{intro:2-RF} the 2-representation-finite algebras which
  are derived equivalent to $\cohX$ and
\item \label{intro:selfinjective-ct-can} the selfinjective
  cluster-tilted algebras of canonical type.
\end{enumerate}

The interest in classifying the objects above has its origin in higher
Auslander-Reiten theory which was introduced by Iyama in
\cite{iyama_higher-dimensional_2007}.  As the name suggests, it is a
higher-dimensional analog of classical Auslander-Reiten theory for
finite dimensional algebras. Let $\Lambda$ be a finite dimensional algebra. Higher Auslander-Reiten theory can be
developed in distinguished subcategories of $\mod \Lambda$,  nowadays
called $n$-cluster-tilting subcategories.  A subcategory $\M$ of $\mod
\Lambda$ is an \emph{$n$-cluster-tilting subcategory} if
\begin{align*}
  \M &= \setP{N\in\mod \Lambda}{\Ext_\Lambda^i(-,N)|_{\M}=0\text{ for
    }
    i\in\set{1,\dots,n-1}} \\
  &= \setP{N\in\mod \Lambda}{\Ext_\Lambda^i(N,-)|_\M=0\text{ for }
    i\in\set{1,\dots,n-1}}.
\end{align*}
One of the most remarkable features of higher Auslander-Reiten theory
is the existence of a functor $\tau_n\colon\M\to\M$ together with a
natural isomorphism
\[
\Ext_\Lambda^n(X,Y) \cong D\overline{\Hom}_\Lambda(Y,\tau_nX)
\quad\text{for all}\quad X,Y\in\M,
\]
which is a higher analog of usual Auslander-Reiten duality.

The simplest class of algebras which have an $n$-cluster-tilting
subcategory are the so-called $n$-representation-finite algebras,
which where introduced by Iyama and Oppermann in
\cite{iyama_n-representation-finite_2011}.  A finite dimensional
algebra $\Lambda$ is said to be \emph{$n$-representation-finite} if
$\Lambda$ has global dimension $n$ and there exists a $\Lambda$-module
$M$ such that $\add M$ is an $n$-cluster-tilting subcategory (in this
case $M$ is called a \emph{$n$-cluster-tilting module}).  For example,
1-representation-finite algebras are precisely representation-finite
hereditary algebras.  In this sense, $n$-representation-finite
algebras may be regarded as a higher analog of representation-finite
hereditary algebras.

Now we explain what are the objects that we classify in this article,
and how do they relate to each other.  The 1-representation-finite
algebras were classified by Gabriel in
\cite{gabriel_unzerlegbare_1972}: they are precisely the algebras
which are Morita-equivalent to the path algebras of quivers whose
underlying graph is a Dynkin diagram of simply-laced type (we work over a fixed algebraically closed field). It is then
natural to study 2-representation-finite algebras. Important
structural results regarding 2-representation finite algebras in terms
of selfinjective quivers with potential have been obtained by
Herschend and Iyama in \cite{herschend_selfinjective_2011} where they
also have provided large classes of examples of such algebras.
Following \cite{happel_tilting_1996}, we say that a finite dimensional
algebra is \emph{piecewise hereditary} if it is derived equivalent to
a hereditary category $\HH$ or, equivalently, if it is isomorphic to
the endomorphism algebra of a tilting complex in $\Db(\HH)$.  From a
homological point of view, the simplest kind of
2-representation-finite algebras are the ones which are piecewise
hereditary.

By a celebrated result of Happel
\cite[Thm. 3.1]{happel_characterization_2001}, it is known that there
are only two kinds of hereditary categories (satisfying suitable
finiteness conditions) which have a tilting object: the ones which are
derived equivalent to $\mod H$ where $H$ is a finite dimensional
hereditary algebra, and the ones which are derived equivalent to
$\cohX$ where $\XX$ is a weighted projective line. We distinguish
between piecewise hereditary algebras as follows: We
say that a finite dimensional algebra $\Lambda$ is \emph{iterated
  tilted} if $\mod \Lambda$ is derived equivalent to $\mod H$ where
$H$ is a finite dimensional hereditary algebra.  Similarly, we say
that $\Lambda$ is \emph{derived-canonical} if $\mod \Lambda$ is
derived equivalent to $\cohX$ for some weighted projective line $\XX$.

Taking advantage of Ringel's classification of the selfinjective
cluster-tilted algebras \cite{ringel_self-injective_2008}, the
2-representation-finite algebras which are iterated tilted were
classified by Iyama and Oppermann in
\cite[Thm. 3.12]{iyama_stable_2013}.  Note that these
algebras are derived equivalent to representation-finite
hereditary algebras whose underlying quiver is of Dynkin type $D$.  In
particular, there are no 2-representation-finite algebras which are
derived equivalent to a tame or wild hereditary algebra.

The following result is the main result of this article. It gives a classification of the 2-representation-finite
derived canonical algebras, and thus complements Iyama-Oppermann's
classification \cite[Thm. 3.12]{iyama_stable_2013}.

\begin{theorem}[see Theorem \ref{thm:2-rf-der-can}]
  \label{intro:classification-2RF}
  The complete list of all basic 2-representation-finite
  derived-canonical algebras is given in Figures \ref{fig:2222-2-APR},
  \ref{fig:244-2RF} and \ref{fig:236-2RF}. In this case, the
  corresponding weighted projective line has tubular type
  $(2,2,2,2;\lambda)$, $(2,4,4)$ or $(2,3,6)$.
  \begin{figure}
    \centering
    \includegraphics[scale=0.5]{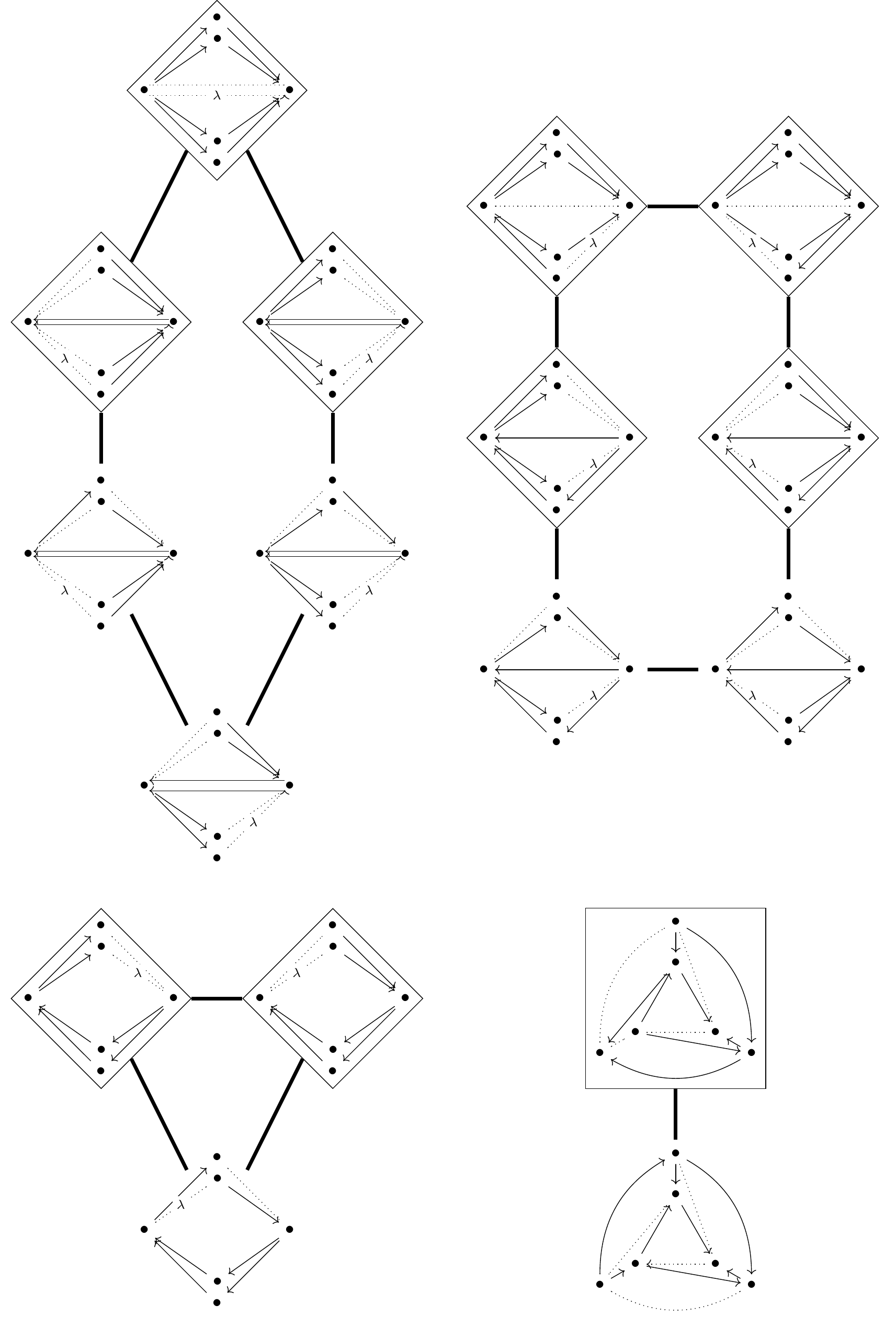}
    \caption{Endomorphism algebras of basic tilting complexes in
      $\Db(\cohX)$ for type $(2,2,2,2;\lambda)$. All complexes are
      $\tau^2$-stable since $\tau^2$ is the identity on $\cohX$. The
      relations are induced by the quivers with potential in Figure
      \ref{fig:2222}; those with label $\lambda$ correspond to relations involving
      the distinguished parameter. Thick lines indicate
      2-APR-(co)tilting. The algebras that arise as
      endomorphism algebras of tilting sheaves in $\cohX$ are enclosed
      in a frame.}
    \label{fig:2222-2-APR}
  \end{figure}
  \begin{figure}
    \centering
     \includegraphics[scale=0.5]{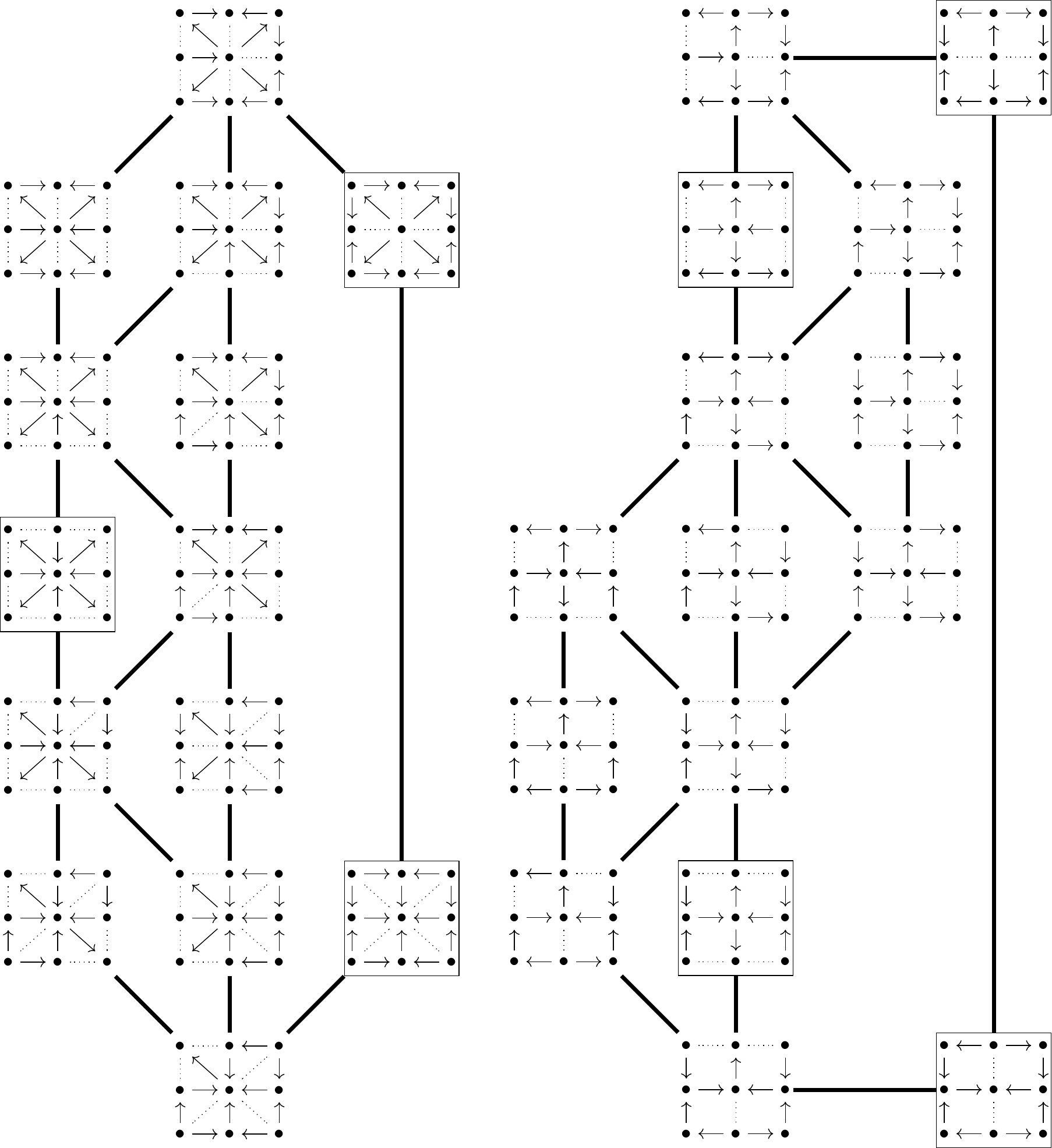} 
    \caption{(Part 1 of 2) The basic 2-representation-finite derived-canonical algebras of
      type (2,4,4). Thick lines indicate 2-APR-(co)tilting. The
      algebras that arise as endomorphism algebras of tilting sheaves
      in $\cohX$ are enclosed in a frame.}
    \label{fig:244-2RF}
  \end{figure}
  \addtocounter{figure}{-1}
  \begin{figure}
    \centering
      \includegraphics[scale=0.5]{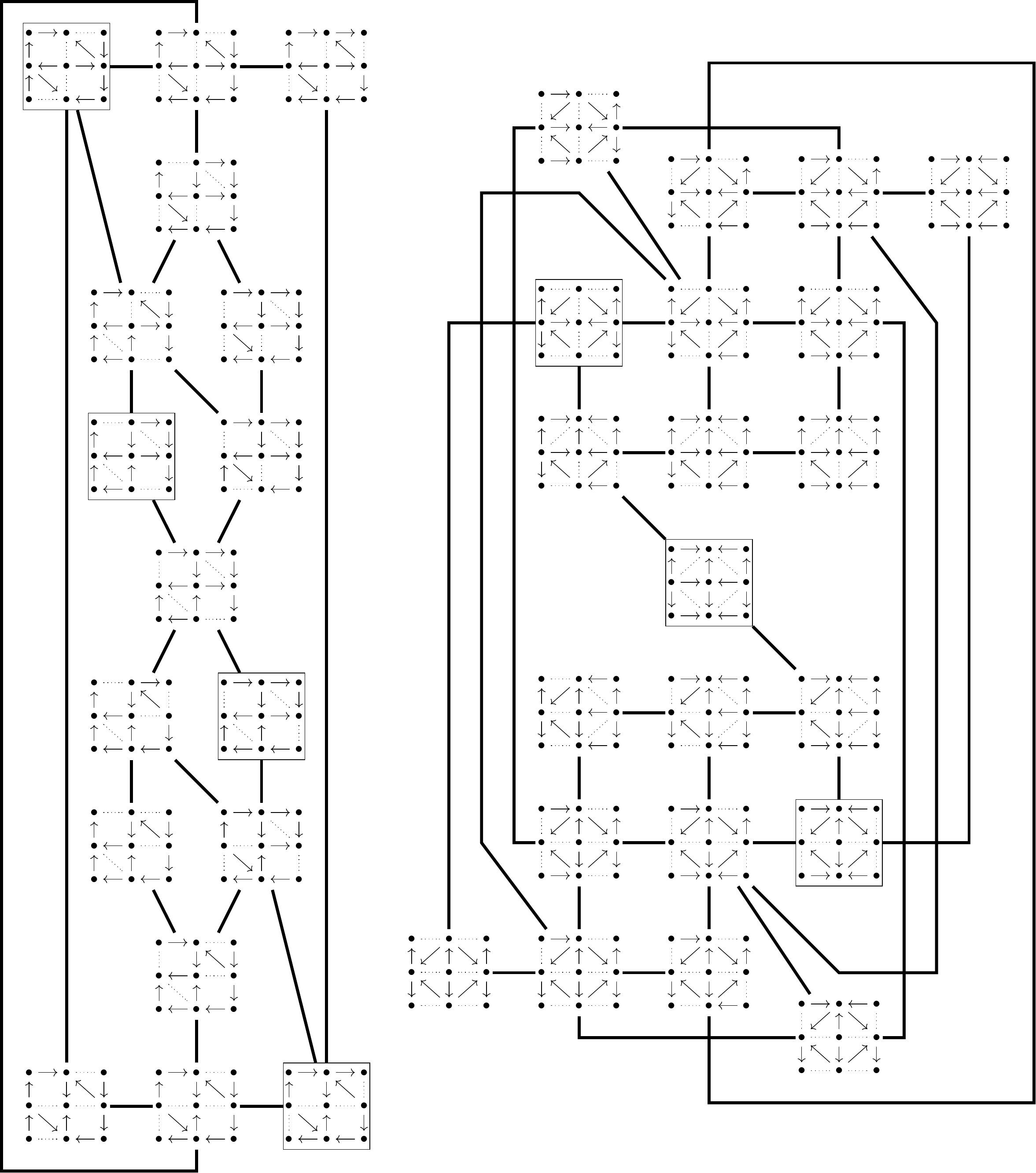}
    \caption{(Part 2 of 2) The basic 2-representation-finite derived-canonical algebras of
      type (2,4,4). Thick lines indicate 2-APR-(co)tilting. The
      algebras that arise as endomorphism algebras of tilting sheaves
      in $\cohX$ are enclosed in a frame.}
  \end{figure}
  \begin{figure}
    \centering
    \includegraphics[scale=0.4]{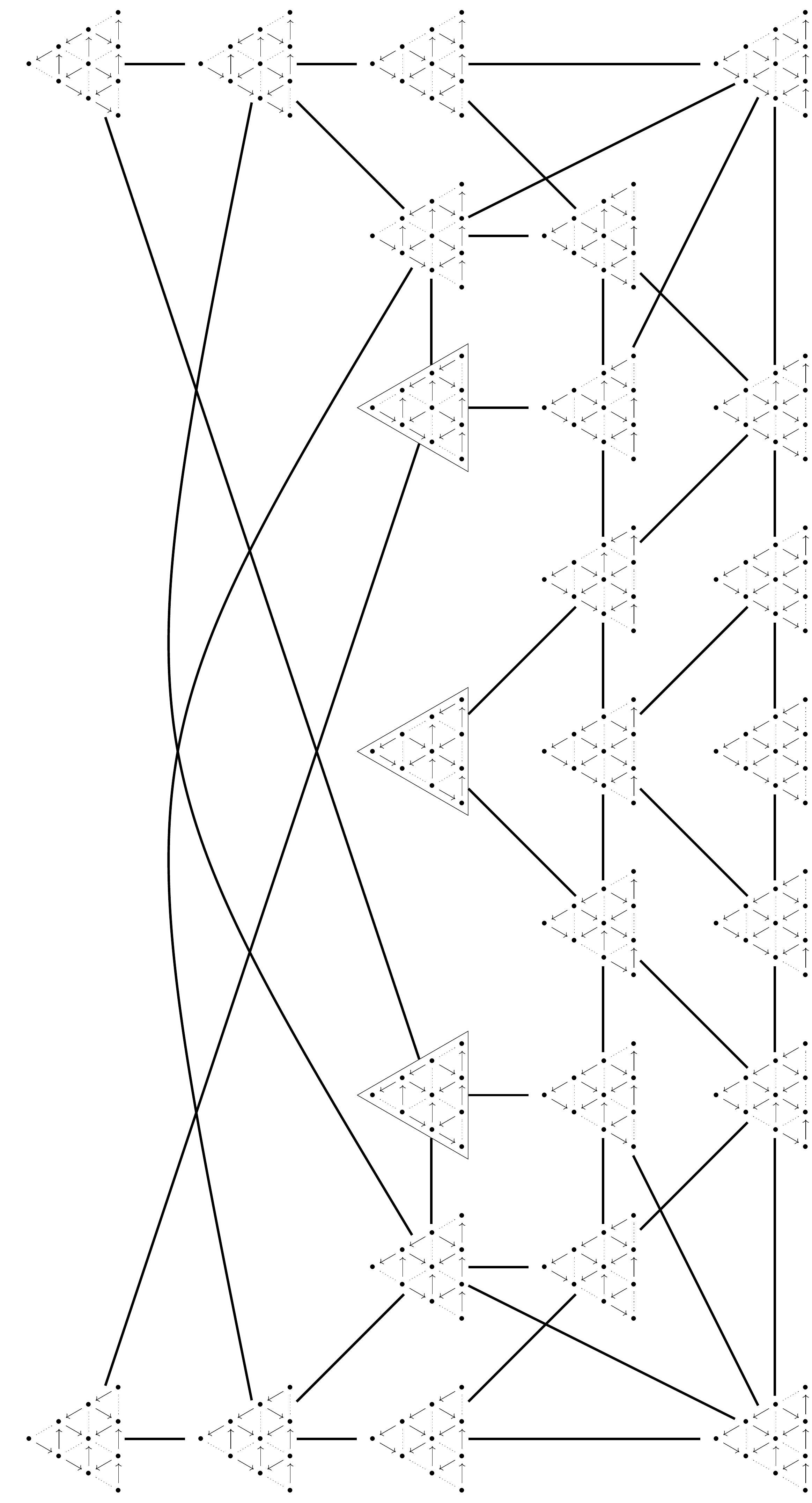}
    \caption{(Part 1 of 2) The basic 2-representation-finite derived-canonical algebras of
      type (2,3,6). Thick lines indicate 2-APR-(co)tilting. The
      algebras that arise as endomorphism algebras of tilting sheaves
      in $\cohX$ are enclosed in a frame.}
    \label{fig:236-2RF}
  \end{figure}
  \addtocounter{figure}{-1}
  \begin{figure}
    \centering
    \includegraphics[scale=0.4]{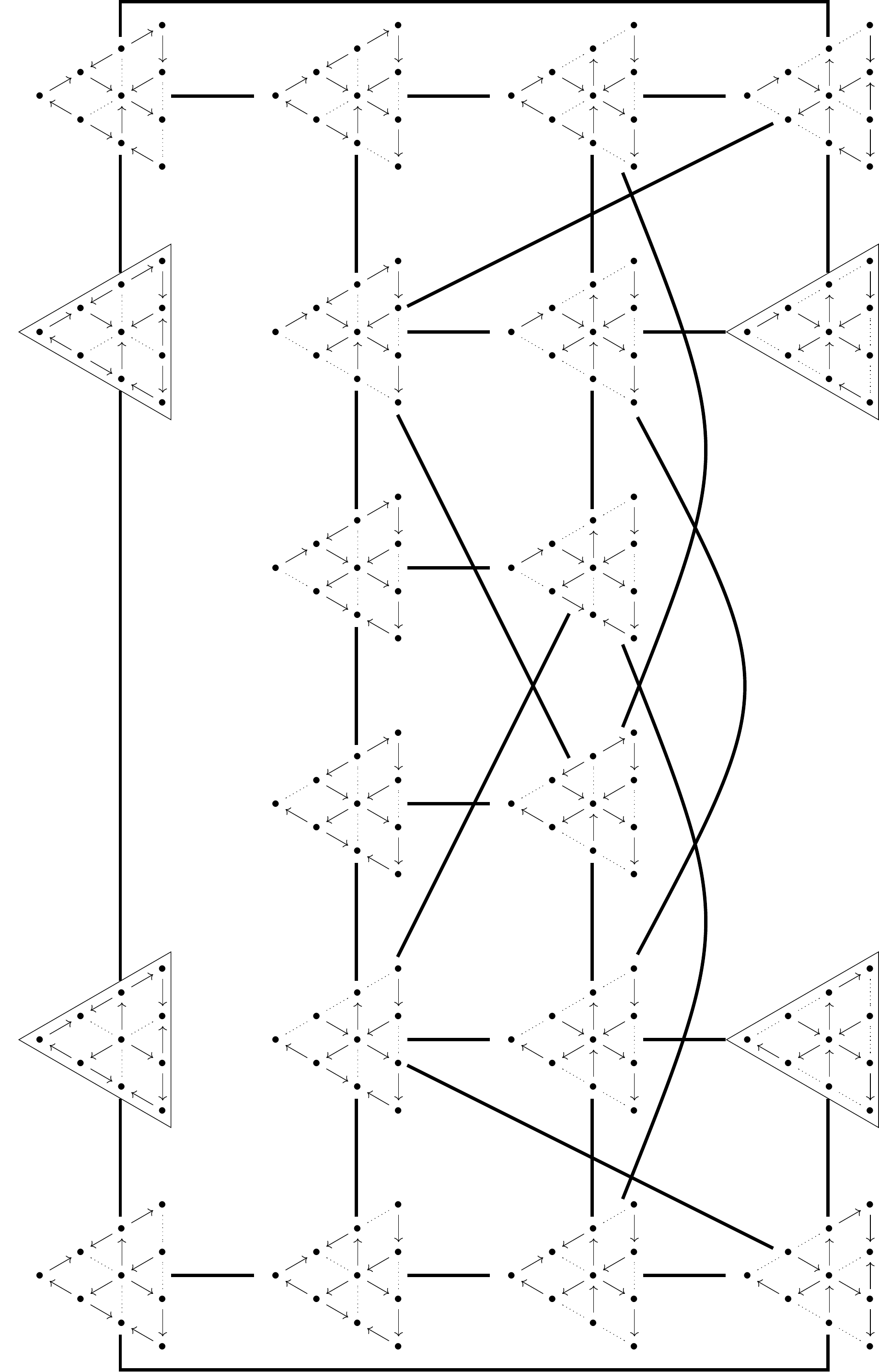}
    \caption{(Part 2 of 2) The basic 2-representation-finite derived-canonical algebras of
      type (2,3,6). Thick lines indicate 2-APR-(co)tilting. The
      algebras that arise as endomorphism algebras of tilting sheaves
      in $\cohX$ are enclosed in a frame.}
  \end{figure}
\end{theorem}

Note that there are no
2-representation-finite algebras which are derived equivalent to
$\cohX$ for a weighted projective line $\XX$ of wild type.  It is
important to note that in the case $(2,2,2,2;\lambda)$ all
derived-canonical algebras are 2-representation-finite. The
classification of all derived-canonical algebras of type
$(2,2,2,2;\lambda)$ is known, see for example Skowro{\'n}ski
\cite[Ex. 3.3]{skowronski_selfinjective_1989}, Barot-de la Pe\~{n}a
\cite[Fig. 1]{barot_derived_1999} and Meltzer in
\cite[Thm. 10.4.1]{meltzer_exceptional_2004}.
Also, part 1 of Figure \ref{fig:244-2RF} already appeared in
\cite[Fig. 1]{herschend_selfinjective_2011}.

We mention that there exists a notion of 2-APR-(co)tilting, which is a
higher analog of classical APR-(co)tilting, and that it preserves
2-representation-finiteness, see Definition \ref{def:2-apr-tilting}.
The algebras in Figures \ref{fig:2222-2-APR},
  \ref{fig:244-2RF} and \ref{fig:236-2RF} are related by
  2-APR-(co)tilting as indicated. 

Let $\tau\colon\Db(\cohX)\to \Db(\cohX)$ be the Auslander-Reiten
translation. We say that a sheaf $X\in\Db(\cohX)$ is
\emph{$\tau^2$-stable} if $\tau^2 X\cong X$. Theorem
\ref{intro:classification-2RF} is a consequence of the following result, which gives a classification of the $\tau^2$-stable tilting
sheaves over a weighted projective line.

\begin{theorem}[see Theorem \ref{thm:t2-cpx}]
  \label{intro:classification-t2}
  Let $\XX$ be a weighted projective line and $T$ a basic tilting
  complex in $\Db(\cohX)$.  Then $T$ is $\tau^2$-stable if and only if
  $\End_{\Db(\XX)}(T)$ is isomorphic to one of the algebras in Figures
  \ref{fig:2222-2-APR}, \ref{fig:244-2-APR} and
  \ref{fig:236-2-APR}. Moreover, this determines $T$ up to an
  autoequivalence of $\Db(\cohX)$. In this case, the
  corresponding weighted projective line has tubular type
  $(2,2,2,2;\lambda)$, $(2,4,4)$ or $(2,3,6)$.
  \begin{figure}
    \centering
    \includegraphics[scale=0.5]{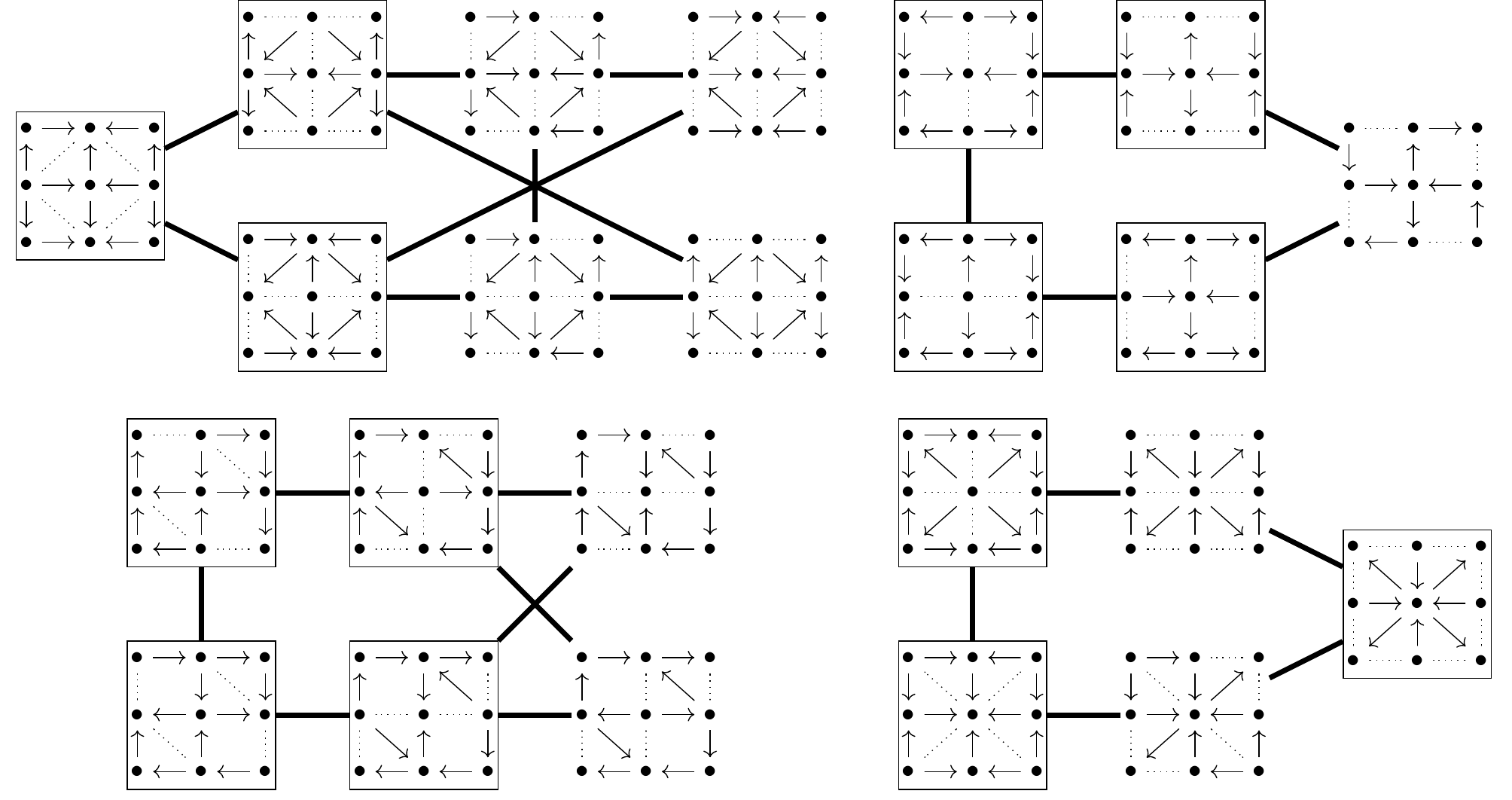}
    \caption{Endomorphism algebras of $\tau^2$-stable basic tilting
      complexes in $\Db(\cohX)$ for type $(2,4,4)$. All relations are
      commutativity or zero relations, \cf Figure \ref{fig:244}.
      Thick lines indicate 2-APR-(co)tilting along orbits of the
      action of $\tau^2$, which is given by rotation by $\pi$.  The
      algebras that arise as endomorphism algebras of tilting sheaves
      in $\cohX$ are enclosed in a frame.}
    \label{fig:244-2-APR}
  \end{figure}
  \begin{figure}
    \centering 
    \includegraphics[scale=0.5]{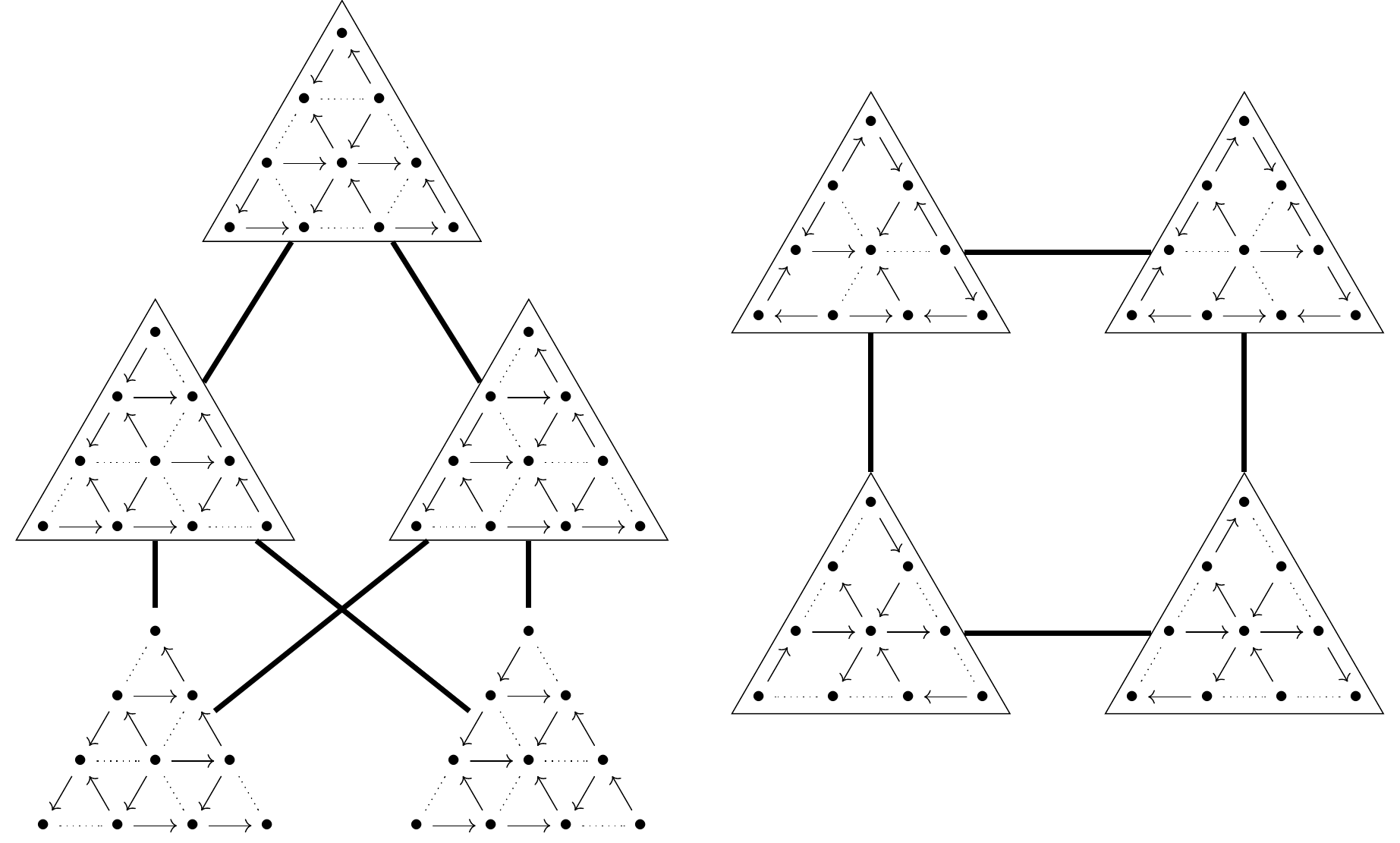}
    \caption{Endomorphism algebras of $\tau^2$-stable basic tilting
      complexes in $\Db(\cohX)$ for type $(2,3,6)$. All relations are
      commutativity or zero relations, \cf Figure \ref{fig:236}.
      Thick lines indicate 2-APR-(co)tilting along orbits of the
      action of $\tau^2$, which is given by counter-clockwise rotation
      by $2\pi/3$.  The algebras that arise as endomorphism algebras
      of tilting sheaves in $\cohX$ are enclosed in a frame.}
    \label{fig:236-2-APR}
  \end{figure}
\end{theorem}

A finite dimensional algebra is \emph{cluster-tilted of canonical
  type} if it is isomorphic to the endomorphism algebra of a
cluster-tilting object in the cluster category $\C_\XX$ associated to
a weighted projective line $\XX$, see Section
\ref{sec:cluster-category} for definitions.

By results of Keller \cite{keller_deformed_2011} and Amiot
\cite{amiot_cluster_2009}, the basic cluster-tilted algebras of
canonical type are 3-preprojective algebras of basic derived canonical
algebras of global dimension at most 2. Moreover, they are Jacobian
algebras of quivers with potential, see Section
\ref{sec:cluster-category}.  As a consequence of Theorem
\ref{intro:classification-2RF}, we obtain a classification of the
selfinjective cluster-tilted algebras of canonical type.  This
complements Ringel's classification \cite{ringel_self-injective_2008}.

\begin{theorem}[see Theorem \ref{thm:classification-ct-can}]
  \label{intro:classification-ct-can}
  The complete list of all basic selfinjective cluster-tilted algebras
  of canonical type is given by the Jacobian algebras of the quivers
  with potential in Figures \ref{fig:2222},
  \ref{fig:244} and \ref{fig:236}. In this case, the
  corresponding weighted projective line has tubular type
  $(2,2,2,2;\lambda)$, $(2,4,4)$ or $(2,3,6)$.
  \begin{figure}
    \centering 
    \includegraphics[scale=1]{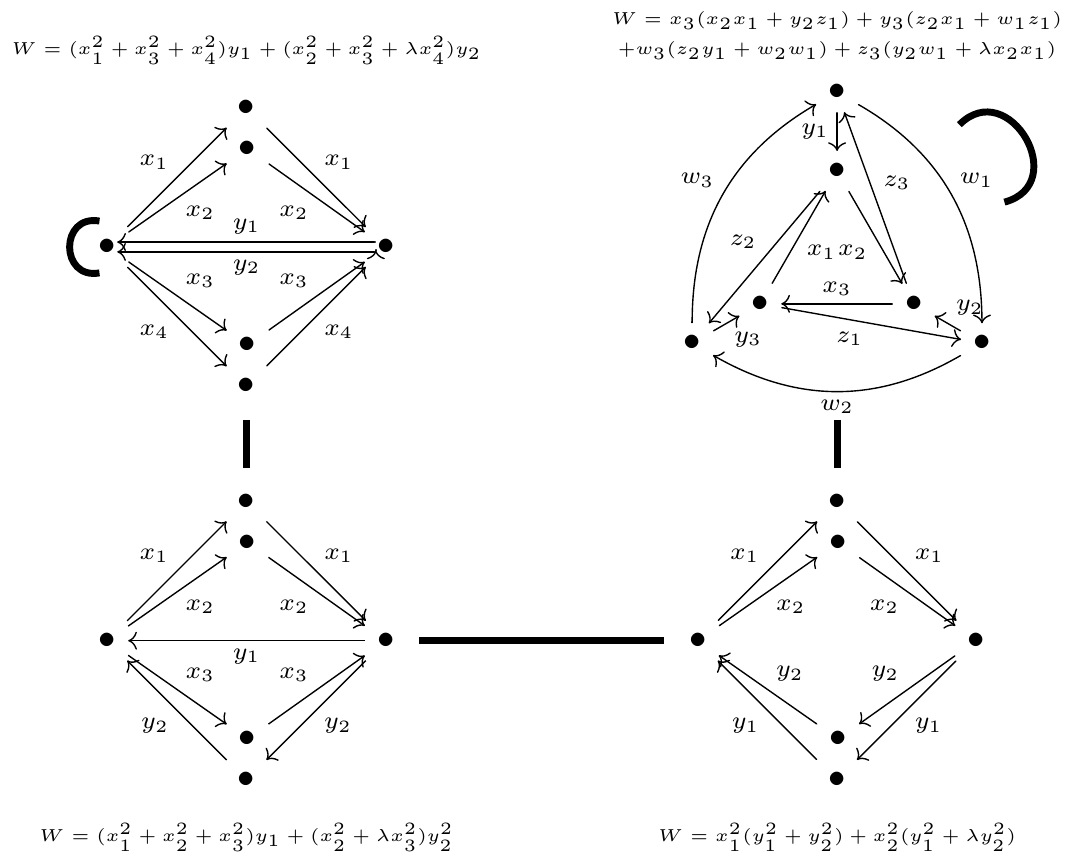}
    \caption{The quivers with potential associated to the basic
      selfinjective cluster-tilted algebras of type
      $\p=(2,2,2,2;\lambda)$. All cluster-tilted algebras are
      selfinjective since $\tau^2=1_\XX$. Thick edges indicate
      mutation of quivers with potential along the orbits of the
      Nakayama permutation, which is trivial in this case.  Note that
      $\lambda\neq 0,1$, and that we may replace $\lambda$ by
      $1-\lambda$, $\frac{1}{\lambda}$, $\frac{1}{1-\lambda}$,
      $\frac{\lambda}{1-\lambda}$ or $\frac{\lambda-1}{\lambda}$
      without changing the isomorphism class of the associated
      Jacobian algebra.}
    \label{fig:2222}
  \end{figure}
  \begin{figure}
    \centering
    \includegraphics[scale=1]{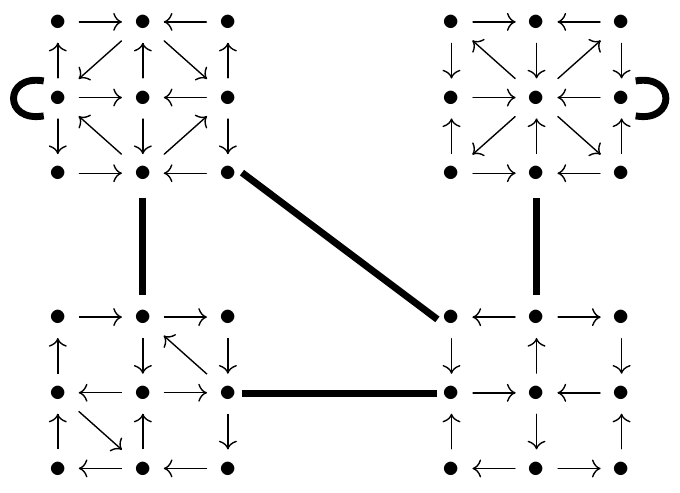}
    \caption{The quivers with potential associated to the basic
      selfinjective cluster-tilted algebras of type $\p=(2,4,4)$. For
      each quiver, the potential is given by $W=\sum \text{(clockwise
        cycles)} - \sum \text{(counter-clockwise cycles)}$. Thick
      edges indicate mutation of quivers with potential along the
      orbits of the Nakayama permutation, which is given by rotation
      by $\pi$.}
    \label{fig:244}
  \end{figure}
  \begin{figure}
    \centering
    \includegraphics[scale=1]{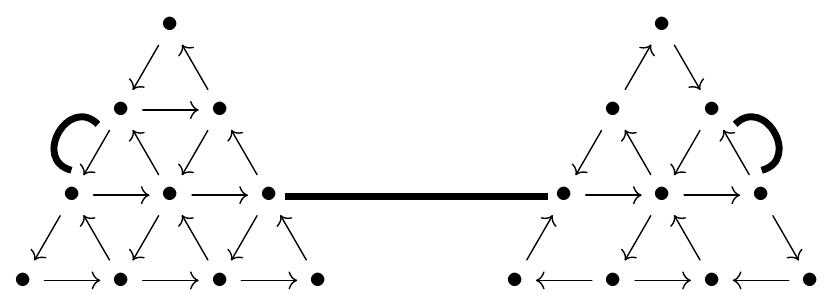}
    \caption{The quivers with potential associated to the basic
      selfinjective cluster-tilted algebras of type $\p=(2,3,6)$. For
      each quiver, the potential is given by $W=\sum \text{(clockwise
        cycles)} - \sum \text{(counter-clockwise cycles)}$.  Thick
      edges indicate mutation of quivers with potential along the
      orbits of the Nakayama permutation, which is given by
      counter-clockwise rotation by $2\pi/3$.}
    \label{fig:236}
  \end{figure}
\end{theorem}

The algebras listed in Theorem \ref{intro:classification-ct-can}
already appeared in related contexts: Figure \ref{fig:2222} is
precisely the exchange graph of endomorphism algebras of
cluster-tilting objects the cluster category associated to a weighted
projective line of type $(2,2,2,2;\lambda)$, see
\cite{barot_tubular_2012}. In addition, Figures \ref{fig:244} and
\ref{fig:236} appeared in \cite[Figs. 3 and
2]{herschend_selfinjective_2011} respectively.

We conclude this section by fixing our conventions and notation.
Throughout this article we work over an algebraically closed field
$K$.  If $\Lambda$ is a finite dimensional $K$-algebra, we denote by
$\mod\Lambda$ the category of finitely generated right
$\Lambda$-modules.  If $\Lambda$ is a basic algebra, we denote its
Gabriel quiver by $Q_\Lambda$.  More generally, if $X$ is a basic
object in a Krull-Schmidt $K$-linear category $\A$, we denote by $Q_X$
the Gabriel quiver of the algebra $\End_\A(X)$.  If $\A$ is an abelian
category, we denote by $\Db(\A)$ the bounded derived category of $\A$
and we identify $\A$ with the full subcategory of $\Db(\A)$ given by
the complexes concentrated in degree zero.

\section{Preliminaries}

\subsection{Coherent sheaves over a weighted projective line}

We recall the construction of the category of coherent sheaves over a
weighted projective line together with its basic properties.  We
follow the exposition of \cite{lenzing_weighted_2011}.

Choose a \emph{parameter sequence}
$\blambda=(\lambda_1,\dots,\lambda_t)$ of pairwise distinct points of
$\P^1_K$ and a \emph{weight sequence} $\p=(p_1,\dots,p_t)$ of positive
integers.  Without loss of generality, we assume that $t\geq3$ and
that for each $i\in\set{1,\dots,t}$ we have $p_i\geq 1$. Moreover, we
may also assume that $\lambda_1=\infty$, $\lambda_2=0$ and
$\lambda_3=1$.  For convenience, we set $p:=\lcm(p_1,\dots,p_t)$. We
call the triple $\XX=(\P^1_K, \blambda,\p)$ a \emph{weighted
  projective line of weight type $\p$}.

The category $\cohX$ of \emph{coherent sheaves over $\XX$} is defined
as follows. Consider the rank 1 abelian group $\LL=\LL(\p)$ with
generators $\vecx_1,\dots,\vecx_t,\vecc$ subject to the relations
\[
p_1\vecx_1=\cdots=p_t\vecx_t=\vecc.
\]
The element $\vecc$ is called the \emph{canonical element of $\LL$}.
It follows that every $\vecx\in\LL$ can be written uniquely in the
form
\[
\vecx = m\vecc + \sum_{i=1}^t m_i\vecx_i
\]
where $m\in\ZZ$ and $0\leq m_i<p_i$ for each $i\in\set{1,\dots,t}$.
Hence $\LL$ is an ordered group with positive cone $\sum_{i=1}^t\NN
x_i$, and that for every $\vecx\in\LL$ we have either $0\leq \vecx$ or
$\vecx\leq \vecc +\vecw$ where
\[
\vecw:=(t-2)\vecc-\sum_{i=1}^t x_i
\]
is the \emph{dualizing element of $\XX$}.

Next, consider the $\LL$-graded algebra $K[x_1,\dots,x_t]$ where
$\deg{x_i}=\vecx_i$ for each $i\in\set{1,\dots,t}$. When $t=3$, we
write $x=x_1$, $y=x_2$, $z=x_3$ and relabel the generators of $\LL$
accordingly. Let $I=(f_3,\dots,f_t)$ be the homogeneous ideal of
$K[x_1,\dots,x_t]$ generated by all the \emph{canonical relations}
\[
f_i=x_i^{p_i}-\lambda_i'x_2^{p_2}-\lambda_i''x_1^{p_1}.
\]
Consequently, we obtain an $\LL$-graded algebra
$R=R(\blambda,\p):=K[x_1,\dots,x_t]/I$.  Note that the group $\LL$
acts by degree shift on the category $\gr^\LL R$ of finitely generated
$\LL$-graded $R$-modules. Namely, given an $\LL$-graded $R$-module $M$
and $\vecx\in\LL$ we denote by $M(\vecx)$ the $R$-module with grading
$M(\vecx)_{\vecy}:=M_{\vecx+\vecy}$.

Let $\gr^{\LL}R$ be te category of finitely generated $\LL$-graded
$R$-modules. Note that $\LL$ acts on $\gr^\LL R$ by degree shift:
given $\vecx\in\LL$ and $M\in\gr^\LL R$, we define $M(\vecx)\in\gr^\LL
R$ to be the $R$-module with $M$ with new grading
$M(\vecx)_{\vecy}:=M_{\vecx+\vecy}$. The category $\cohX$ is defined
as the localization $\qgr^{\LL}R$ of $\gr^{\LL}R$ by its Serre
subcategory $\gr_0^{\LL}R$ of finite dimensional $\LL$-graded
$R$-modules.  We denote the image of a module $M$ under the
canonical quotient functor $\gr^{\LL}R\to\qgr^{\LL}R$ by
$\widetilde{M}$.  It follows that the action of $\LL$ on $\gr^\LL R$
induces an action on $\cohX$ given by
$\widetilde{M}(\vecx):=(M(\vecx))^{\sim}$.  We call
$\OO=\OO_\XX:=\widetilde{R}$ the \emph{structure sheaf of $\XX$}.

\begin{theorem}
  \label{thm:props-cohX}
  \cite{geigle_class_1985} \cite[Thm. 2.2]{lenzing_weighted_2011} The
  category $\cohX$ is connected, $\Hom$-finite, $K$-linear and
  abelian.  Moreover we have the following:
  \begin{enumerate}
  \item The category $\cohX$ is hereditary, \ie we have
    $\Ext_\XX^i(-,-)=0$ for all $i\geq 2$.
  \item (Serre duality) Let $\tau:\cohX\to\cohX$ be the
    autoequivalence given by $E\mapsto E(\vecw)$. Then, there is a
    bifunctorial isomorphism
    \[
    D\Ext_\XX^1(X,Y) \cong \Hom_\XX(Y,\tau X).
    \]
    We call $\tau$ the \emph{Auslander-Reiten translation} of $\cohX$.
  \item\label{it:decomposition} Let $\coh_0\XX$ be the full
    subcategory of $\coh\XX$ of sheaves of finite length
    (=\emph{torsion sheaves}). Also, let $\vect\XX$ be the full
    subcategory of $\cohX$ of sheaves with no non-zero torsion
    subsheaves (=\emph{vector bundles}).  Then, each $X\in\cohX$ has a
    unique decomposition $X=X_+\oplus X_0$ where $X_+\in\vect\XX$ and
    $X_0\in\coh_0\XX$.
  \item\label{it:simples} The simple objects in $\coh_0\XX$ are
    parametrized by $\P^1_K$ as follows: For each
    $\lambda\in\P^1_K\setminus\blambda$ there exist a unique simple
    sheaf $S_\lambda$ called the \emph{ordinary simple concentrated at
      $\lambda$}, and for each $\lambda_i\in\blambda$ there exist
    $p_i$ \emph{exceptional} (\ie not ordinary) simple sheaves
    $S_{\lambda,1},\dots,S_{\lambda,p_i}$ defined by a short exact
    sequence
    \[
    \begin{tikzcd}
      0\rar&\OO(-m\vecx_i)\rar&\OO((1-m)\vecx_i)\rar& S_{\lambda_i,m}\rar&0
    \end{tikzcd}
    \]
    for $i\in\set{1,\dots,t}$ and $m\in\set{1,\dots,p_i}$.
  \item For each simple sheaf $S$ we have $\End_\XX(S)\cong K$. If $S$
    is an ordinary simple sheaf, then $\Ext_\XX^1(S,S)\cong K$. If $S$
    is an exceptional simple sheaf, then $\Ext_\XX^1(S,S)=0$.
  \item Let $\lambda\in\P^1_K$. The category $\T(\lambda)$ of all
    sheaves which have a finite filtration by simple sheaves
    concentrated at $\lambda$ form a standard tube. If
    $\lambda\notin\blambda$ then $\T(\lambda)$ has rank 1; if
    $\lambda=\lambda_i$, then $\T(\lambda)$ has rank $p_i$ .
  \item\label{it:morphisms} Let $\vec{a},\vec{b}\in\LL$. Then
    $\Hom_\XX(\OO(\vec{a}),\OO(\vec{b}))=R_{\vec{b}-\vec{a}}$. In
    particular, there is a non-zero morphism
    $\OO(\vec{a})\to\OO(\vec{b})$ if and only if
    $\vec{b}-\vec{a}\geq0$.
  \end{enumerate}
\end{theorem}

The complexity of the classification of indecomposable sheaves $\cohX$
is controlled by its \emph{Euler characteristic}
\[
\chi(\XX):=2-\sum_{i=1}^t\left(1-\frac{1}{p_i}\right).
\]
Weighted projective lines of Euler characteristic zero will turn out
to be our main concern in this article.  An easy calculation shows
that $\chi(\XX)=0$ if and only if
\[
\mathbf{p}\in\set{(2,2,2,2),(3,3,3),(2,4,4),(2,3,6)},
\]
if and only if the dualizing element $\vecw$ has finite order
$p=\lcm(p_1,\dots,p_t)$ in $\LL$.  In this case, we say that $\XX$ has
\emph{tubular type} and it follows that $\tau$ acts periodically on
each connected component of the Auslander-Reiten quiver of $\cohX$.

Let $\K_0(\XX)$ be the Grothendieck group of $\cohX$.  There are two
important linear forms $\rk$ and $\deg$ on $\K_0(\XX)$.  We refer the
reader to \cite[Sec. 2.2]{lenzing_weighted_2011} for information on
these numerical invariants.  The rank $\rk\colon\K_0(\XX)\to\ZZ$ is
characterized by the property $\rk(\OO(\vecx))=1$ for each $\vecx$ in
$\LL$.  The degree $\deg\colon\K_0(\XX)\to\ZZ$ is characterized by the
property $\deg(\OO(\vecx))=\delta(\vecx)$ where
$\delta\colon\LL\to\ZZ$ is the unique group homomorphism sending each
$\vecx_i$ to $p/p_i$.  Note that we have
\begin{equation}
  \label{eq:delta-euler}
  \chi(\XX)=\frac{-\delta(\vecw)}{p}.
\end{equation}
Moreover, if a sheaf $X\in\cohX$ satisfies $\deg(X)=\rk(X)=0$, then
$X=0$. The slope of a non-zero sheaf $X$is defined as
$\slope(X):=\rk(X)/\deg(X)\in\QQ\cup\set{\infty}$.

\begin{proposition}
  \label{prop:slope}
  \cite[Lemma. 2.5]{lenzing_weighted_2011} For each non-zero $X\in\vect X$ we
  have $\slope(\tau X) = \slope(X) + \delta(\vecw)$.
\end{proposition}

A complex $T$ in $\Db(\cohX)$ is called a \emph{tilting complex} if
$\Ext_{\Db(\XX)}^i(T,T)=0$ for all $i\neq0$ and if the conditions
$\Ext_{\Db(\XX)}^i(T,X)=0$ for all $i\in\ZZ$ imply that
$X=0$. Equivalently, $T$ is a tilting complex if and only if $T$ is
rigid, \ie $T$ has no non-zero self-extensisons, and the number of pairwise non-isomorphic indecomposable direct
summands of $T$ equals $2+\sum_{i=1}^t (p_i-1)$, the rank of
$\K_0(\XX)$.

The vector bundle
\[
T=T_{\OO} := \bigoplus_{0\leq\vecx\leq\vecc} \OO(\vecx)
\]
is a tilting sheaf whose endomorphism algebra is precisely
$\Lambda=\Lambda(\blambda,\p)$, the \emph{canonical algebra of type
  $(\blambda,\p)$}, see Figure \ref{fig:canonical-algebra-2222} for an
example.
\begin{figure}
  \begin{center}
    \begin{tikzpicture}[commutative diagrams/every diagram]
      \node (a) at (180:2) {$\OO$}; \node (b) at (90:2)
      {$\OO(\vecx_1)$}; \node (c) at (0:2) {$\OO(\vecc)$}; \node (d)
      at (90:1) {$\OO(\vecx_2)$}; \node (e) at (90:-1)
      {$\OO(\vecx_3)$}; \node (f) at (90:-2) {$\OO(\vecx_4)$};
      \path[commutative diagrams/.cd, every arrow, every label] (a)
      edge node {$x_1$} (b) (b) edge node {$x_1$} (c) (a) edge
      node[swap] {$x_2$} (d) (d) edge node[swap] {$x_2$} (c) (a) edge
      node {$x_3$} (e) (e) edge node {$x_3$} (c) (a) edge node[swap]
      {$x_4$} (f) (f) edge node[swap] {$x_4$} (c) ;
    \end{tikzpicture}
    \caption{The Gabriel quiver of the endomorphism algebra of the
      canonical tilting bundle for $\p=(2,2,2,2;\lambda)$.}
    \label{fig:canonical-algebra-2222}
  \end{center}
\end{figure}
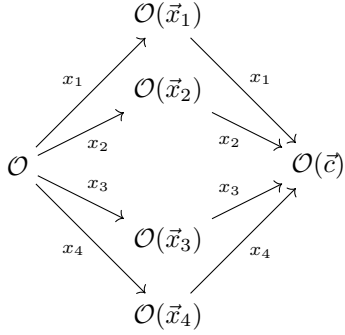
It follows that the bounded derived categories $\Db(\cohX)$ and
$\Db(\mod\Lambda)$ are equivalent as triangulated categories.


\begin{proposition}
  \label{prop:normal-position}
  \cite[Cor. 3.5]{lenzing_tilting_1996}. Let $\XX$ be a weighted
  projective line of tubular type and $T$ a tilting sheaf in
  $\cohX$. Then there exists an automorphism
  $F\colon\Db(\cohX)\to\Db(\cohX)$ such that $FT\in\cohX$ has a simple
  sheaf as a direct summand.
\end{proposition}

Let $\XX$ be a weighted projective line of tubular type and $T$ a
tilting sheaf in $\cohX$. We say that $T$ is in \emph{normal position}
if $T_0\neq0$, see Theorem
\ref{thm:props-cohX}\eqref{it:decomposition} and Theorem
\ref{prop:normal-position}.

The result below collects further properties of $\cohX$ which are
needed to prove Theorem \ref{thm:classification-tau2X}.

\begin{theorem}
  \label{thm:techincal-general}
  \cite{geigle_perpendicular_1991,lenzing_weighted_2011} Let $\XX$ be
  a weighted projective line of weight type $(p_1,\dots,p_t)$ and $T$
  a tilting sheaf in normal position. Then, the following statements
  hold:
  \begin{enumerate}
  \item\label{it:perpendicular-category} Let $q_i$ be the number of
    indecomposable direct summands of $T_0$ in $\T(\lambda_i)$. Then,
    the perpendicular category
    \[
    T_0^\perp:=\setP{X\in\cohX}{\Hom_\XX(T_0,X)=0\text{ and
      }\Ext_\XX^1(T_0,X)=0}
    \]
    is equivalent to $\coh\YY$ where $\YY$ is a weighted projective
    line of weight type $(p_1-q_1,\dots,p_t-q_t)$.
  \item\label{it:Y-H} If $\XX$ has tubular type, then
    $\chi(\YY)>0$. In this case, $\coh\YY$ is derived equivalent to
    $\mod H$ for a tame heredirary algebra of extended Dynkin type
    $\Delta$, and the Auslander-Reiten quiver of $\vect\YY$ has shape
    $\ZZ\Delta$.
  \item\label{it:embedding-preserves-vect-coh0} The embedding
    $\coh\YY\cong T_0^\perp\subset\cohX$ preserves line bundles and
    torsion sheaves. That is, we have $\vect\YY\cong(\vect\XX\cap
    T_0^\perp)$ and $\coh_0\YY\cong(\coh_0\XX\cap T_0^\perp)$.
  \item\label{it:action-vecY} Let $\vecx\in\LL$ be such that
    $T(\vecx)\cong T$. Then the functor $?(\vecx)\colon\cohX\to\cohX$
    induces an action on $\coh\YY$ which acts freely on line bundles
    in $\coh\YY$.
  \item\label{it:line-bundle-as-summand} The sheaf $T_+$ is a tilting
    bundle in $\coh\YY$. If $\XX$ has tubular type, then $T_+$
    contains a line bundle as a direct summand.
  \end{enumerate}
\end{theorem}
\begin{proof}
  Statements \eqref{it:perpendicular-category} and
  \eqref{it:embedding-preserves-vect-coh0} are shown in
  \cite[Thm. 9.5 and Prop. 9.6]{geigle_perpendicular_1991}.

  \eqref{it:Y-H} The first claim is a straightforward computation. The
  remaining statements are shown for example in \cite[Thm. 3.5,
  Cor. 3.6]{lenzing_weighted_2011}.

  \eqref{it:action-vecY} First, note that the group $\LL$ acts freely
  on line bundles in $\cohX$ be degree shift. Moreover, this action
  preserves $\vect\XX$ and $\coh_0\XX$. Let $\vecx\in\LL$ be such that
  $T(\vecx)\cong T$. Since we have $T_0(\vecx)\cong T_0$, it follows
  that $(\vecx)$ induces an action on $T_0^\perp\cong\coh\YY$. By part
  \eqref{it:embedding-preserves-vect-coh0} this action acts freely on
  line bundles in $\coh\YY$.
  
  \eqref{it:line-bundle-as-summand} The first claim follows since
  $T_+$ is also rigid in $T_0^\perp\subset\cohX$ and the number of
  indecomposable direct summands of $T_+$ coincides with the rank of
  $\K_0(\YY)$.  The second claim follows since $\chi(\YY)>0$, and
  hence every tilting bundle in $\coh\YY$ contains a line bundle as a
  direct summand, see \cite[Cor. 3.7]{lenzing_weighted_2011}.
\end{proof}

We have the following simple oveservation regarding $\tau^2$-stable
rigid sheaves.

\begin{lemma}
  \label{lemma:T0-semisimple}
  Let $\XX$ be a weighted projective line and $X\in\cohX$ be a
  $\tau^2$-stable rigid sheaf. Then, each indecomposable direct
  summand of $X$ is an exceptional simple sheaf.
\end{lemma}
\begin{proof}
  Firstly, by Theorem \ref{thm:props-cohX}\eqref{it:morphisms} there
  are no rigid sheaves in an exceptional tube of rank 1. Secondly,
  since $X$ is a rigid $\tau^2$-stable sheaf, we have that
  \[
  \Ext_\XX^1(X,X)\cong D\Hom_\XX(X,\tau X)\cong D\Hom_\XX(\tau X,
  X)=0.
  \]
  Let $Y$ be an indecomposable direct summand of $X$. Then we have
  $\Hom_\XX(\tau Y,Y)=0$. This is happens if and only if $Y$ is an
  exceptional simple sheaf.
\end{proof}

The Auslander-Reiten translation of $\cohX$ extends to an
autoequivalence
\[
\tau\colon\Db(\cohX)\to\Db(\cohX).
\]
Moreover, the autoequivalence $\nu:=\tau[1]$ gives a Serre functor of
$\Db(\cohX)$.

\begin{definition}
  A complex $X$ in $\Db(\cohX)$ is \emph{$\tau^2$-stable} if
  $\tau^2X\cong X$.
\end{definition}

The following result is a particular case of
\cite[Thm. 3.1]{lenzing_tilting_1996}. It allows us to compute the
endomorphism algebra of a tilting sheaf in a given weighted projective
line in terms of a weighted projective line of smaller weights.

\begin{theorem}
  \label{thm:multiextension}
  \cite[Thm. 3.1]{lenzing_tilting_1996} Let $\XX$ be a weighted
  projective line of type $(p_1,\dots,p_t)$. Let $T$ be a
  $\tau^2$-stable tilting sheaf in $\cohX$, and suppose that the
  indecomposable direct summands of $T_0$ are exceptional simple
  sheaves at the points
  $\lambda_{i_i},\dots,\lambda_{i_k}\in\blambda$. We make the
  identification $\coh\YY\cong T_0^\perp$, see Proposition
  \ref{thm:techincal-general}\eqref{it:perpendicular-category}.
  Finally, let $E\in\coh\YY$ be the direct sum of all exceptional
  simple sheaves at the points
  $\lambda_{i_1},\dots,\lambda_{i_k}$. Then, there is an isomorphism
  of algebras
  \[
  \End_\XX(T)\cong\End_\YY(T_+\oplus E)\cong
  \begin{bmatrix}
    \End_\YY(T_+) & \Hom_\YY(T_+,E) \\
    0 & \End_\YY(E)
  \end{bmatrix}.
  \]
\end{theorem}
\begin{proof}
  Let $r\colon \cohX\to \coh\YY$ be the right adjoint of the inclusion
  $\coh\YY\cong T_0^\perp$. It is easy to see that $r$ induces a
  bijection between the indecomposable direct summands of $T_0$ and
  the exceptional simple sheaves in $\coh\YY$ at the points
  $\lambda_{i_1},\dots,\lambda_{i_k}$. Then the result follows
  immediately from the proof of \cite[Thm. 3.1]{lenzing_tilting_1996}.
\end{proof}

\subsection{Graded quivers with potential and their mutations}
\label{sec:quiver-with-potential}

Quivers with potentials and their Jacobian algebras where introduced
in \cite{derksen_quivers_2008} as a tool to prove several of the
conjectures of \cite{fomin_cluster_2007} about cluster algebras in a
rather general setting, see \cite{derksen_quivers_2010}.  Their graded
version was introduced in \cite{amiot_cluster_2010} in order to
describe the effect of mutation of cluster tilting objects in
generalized cluster categories at the level of the corresponding
derived category.

Let $Q=(Q_0,Q_1)$ be a finite quiver without loops or 2-cycles and
$d:Q_1\to \ZZ$ a map called a \emph{degree function} on the set of
arrows of $Q$.  Then $d$ induces a $\ZZ$-grading on the complete path
algebra $\widehat{kQ}$ in an obvious way. We endow $\widehat{KQ}$ with
the $\mathcal{J}$-adic topology where $\mathcal{J}$ is the radical of
$\widehat{kQ}$.  A \emph{potential} in $Q$ is a formal linear
combination of cyclic paths in $Q$; we are only interested in
potentials which are homogeneous elements of $\widehat{kQ}$.  For a
cyclic path $a_1\cdots a_d$ in $Q$ and $a\in Q_1$, let
\[
\partial_a(a_1\cdots a_d) = \sum_{a_i=a} a_{i+1}\cdots a_da_1\cdots
a_{i-1}
\]
and extend it linearly and continuously to an arbitrary potential in
$Q$.  The maps $\partial_a$ are called \emph{cyclic derivatives}.

\begin{definition}
  A \emph{graded quiver with potential} is a triple $(Q,W,d)$ where
  $(Q,d)$ is a $\ZZ$-graded finite quiver without loops and 2-cycles
  and $W$ is a homogeneous potential for $Q$.  The \emph{graded
    Jacobian algebra} of $(Q,W,d)$ is the $\ZZ$-graded algebra
  \[
  \Jac(Q,W,d) \cong \dfrac{\widehat{kQ}}{\partial(W)}
  \]
  where $\partial(W)$ is the closure in $\widehat{kQ}$ of the ideal
  generated by the subset 
  \[
  \setP{\partial_a(W)}{a\in Q_1}.
  \]
\end{definition}

For each vertex of $Q$ there is a pair of well defined operations on
the right-equivalence classes of graded quivers with potential called
\emph{left and right mutations} (see
\cite[Def. 4.2]{derksen_quivers_2008} for the definition of
right-equivalence).  Note that right equivalent quivers with potential
have isomorphic Jacobian algebras.

Let $(Q,W,d)$ be graded quiver with potential with $W$ homogeneous of
degree $d(W)$ and $k\in Q_0$.  The \emph{non-reduced left mutation at
  $k$} of $(Q,W,d)$ is the graded quiver with potential
$\tilde{\mu}_k^L(Q,W,d)=(Q',W',d')$ defined as follows:
\begin{enumerate}
\item The quivers $Q$ and $Q'$ have the same set of vertices.
\item All arrows of $Q$ which are not adjacent to $k$ are also arrows
  of $Q'$ and of the same degree.
\item Each path $i\xto{a}k\xto{b}j$ in $Q$ creates an arrow $[ba]:i\to
  j$ of degree $d(a)+d(b)$ in $Q'$.
\item Each arrow $a:i\to k$ of $Q$ is replaced in $Q'$ by an arrow
  $a^*:k\to i$ of degree $-d(a)+d(W)$.
\item Each arrow $b:k\to j$ of $Q$ is replaced in $Q'$ by an arrow
  $b^*:j\to k$ of degree $-d(b)$.
\item The new potential is given by
  \[ W' = [W] + \sum_{i\xto{a}k\xto{b}j} [ba]a^* b^*\] where $[W]$ is
  the potential obtained from $W$ by replacing each path
  $i\xto{a}k\xto{b}j$ which appears in $W$ with the corresponding
  arrow $[ba]$ of $Q'$.
\end{enumerate}
By \cite[Thm. 6.6]{amiot_cluster_2010}, there exists a graded quiver
with potential $(Q'_{\mathrm{red}},W'_{\mathrm{red}}, d'_{\mathrm{red}})$ which is
right equivalent to $(Q',W',d')$ and such that $W'_{\mathrm{red}}$ only involves cycles of length greater or equal than 3 in $Q_{\mathrm{red}}'$.  The \emph{left mutation at $k$} of
$(Q',W',d')$ is then defined as
\[
\mu_k^L(Q,W,d):= (Q'_{\mathrm{red}},W'_{\mathrm{red}}, d').
\]
The \emph{right mutation at $k$} of $(Q,W,d)$ is defined almost
identically, just by replacing (d) and (e) above by
\begin{enumerate}
\item[(d)] Each arrow $a:i\to k$ of $Q$ is replaced in $Q'$ by an
  arrow $a^*:k\to i$ of degree $-d(a)$.
\item[(e)] Each arrow $b:k\to j$ of $Q$ is replaced in $Q'$ by an
  arrow $b^*:j\to k$ of degree $-d(b)+d(W)$.
\end{enumerate}
Finally, the following definition is very convenient for our purposes.

\begin{definition}
  \label{def:truncated-Jacobian-algebra}
  \cite[Sec. 3]{herschend_selfinjective_2011} Let $(Q,W,d)$ be a
  graded quiver with potential with $d(W)=1$.  Then the \emph{truncated
    Jacobian algebra} is the degree zero part of $\Jac(Q,W,d)$, which
  is given by the factor algebra
  \[
  \Jac(Q,W,d):=\Jac(Q,W)/\overline{\langle a\in Q\mid d(a)=1 \rangle}
  =\widehat{kQ}/\overline{\langle \partial_a(W)\mid d(a)=1\rangle}.
  \]
  Also, we say that $(Q,W,d)$ is \emph{algebraic} if $\Jac(Q,W,d)$ has
  global dimension at most 2 and the set
  \[
  \setP{\partial_a(W)}{d(a)=1}
  \]
  is a minimal set of generators of the ideal
  $\overline{\langle \partial_a(W)\in Q\mid d(a)=1 \rangle}$ of
  $\widehat{kQ}$.
\end{definition}
Note that left and right mutation differ from each other at the level
of the grading only.

\subsection{2-representation-finite algebras and 2-APR-tilting}
\label{sec:2-RF}

Let $\Lambda$ be a finite dimensional algebra of global dimension 2.
Following \cite[Def. 2.2]{iyama_n-representation-finite_2011}, we say
that $\Lambda$ is \emph{2-representation-finite} if there exist a
finite dimensional $\Lambda$-module $M$ such that
\[
\add M = \setP{X\in\mod\Lambda}{\Ext_\Lambda^1(M,X)=0} =
\setP{X\in\mod\Lambda}{\Ext_\Lambda^1(X,M)=0}.
\]
Such a $\Lambda$-module $M$ is called a \emph{2-cluster-tilting module}.
The functors
\[
\tau_2:=D\Ext_\Lambda^2(-,\Lambda):\mod\Lambda\to\mod\Lambda
\]
and
\[
\nu_2:=\nu[-2]:\Db(\mod\Lambda)\to \Db(\mod\Lambda),
\]
where $\nu:-\otimes_{\Lambda}^{\mathbf{L}}D\Lambda:\Db(\mod\Lambda)\to
\Db(\mod\Lambda)$ is the Nakayama functor of $\Db(\mod\Lambda)$ play
an important role in the theory of 2-representation-finite algebras.
Moreover, they are related by a functorial isomorphism $\tau_2\cong
H^0(\nu_2 -)$.  Note that $\tau_2$ induces a bijection between
indecomposable non-projective objects in $\add M$ and indecomposable
non-injective objects in $\add M$.

\begin{definition}
  \label{def:2RF}
  \cite[Def. 1.2]{herschend_n-representation-finite_2011} Let
  $\Lambda$ be a 2-representation-finite algebra.  We say that
  $\Lambda$ is \emph{2-homogeneous} if each $\tau_2$-orbit of
  indecomposable objects in $\add M$ consists of precisely two
  objects.  This is equivalent to $\nu_2^{-1}(\Lambda)$ being an
  injective $\Lambda$-module.
\end{definition}

The class of 2-representation-finite algebras can be characterized in
terms of the so-called 3-preprojective algebras.

\begin{definition}
  \cite{keller_deformed_2011} Let $\Lambda$ be a finite dimensional
  algebra of global dimension at most 2.  The \emph{complete
    3-preprojective algebra of $\Lambda$} is the tensor algebra
  \[
  \Pi_3(\Lambda) := \prod_{d\geq0}
  \Ext^2_\Lambda(D\Lambda,\Lambda)^{\otimes d}.
  \]
\end{definition}

We have the following characterization of 2-representation-finite
algebras.

\begin{proposition}
  \label{prop:charactetrization-2RF-selfinjective-QP}
  \cite[Prop. 3.9]{herschend_selfinjective_2011} Let $\Lambda$ be a
  finite dimensional algebra of global dimension at most 2.  Then
  $\Lambda$ is 2-representation-finite if and only if $\Pi_3(\Lambda)$
  is a finite dimensional selfinjective algebra.
\end{proposition}

Following \cite{keller_deformed_2011}, $\Pi_3(\Lambda)$ can be
presented as a graded Jacobian algebra for some quiver with potential
obtained from $\Lambda$.  For this, let $Q$ be the Gabriel quiver of
$\Lambda$ and let
\[
\Lambda \cong \widehat{kQ}/\overline{\langle r_1,\dots r_s\rangle}
\]
where $\set{r_1,\dots, r_s}$ is a minimal set of relations for
$\Lambda$.  Consider the extended quiver
\[
\tilde{Q} = Q\amalg \setP{r_i^*:t(r_i)\to s(r_i)}{
  r_i:s(r_i)\dashrightarrow t(r_i)}_{1\leq i\leq s},
\]
\ie $\tilde{Q}$ is obtained from $Q$ by adding an arrow in the
opposite direction for each relation in $\Lambda$.  We consider
$\tilde{Q}$ as a graded quiver where the arrows in $Q_1$ have degree
zero and the arrows $r_i^*$ have degree one.  Then we can define a
homogeneous potential $W$ in $\tilde{Q}$ of degree one by
\[
W:= \sum_{i=1}^s r_ir_i^*.
\]

\begin{theorem}
  \label{thm:tildeLambda}
  \cite[Thm. 6.10]{keller_deformed_2011} Let $\Lambda$ be a finite
  dimensional algebra of global dimension at most 2.  Then there is an
  isomorphism of graded algebras between $\Jac(\tilde{Q},W,d)$ and
  $\Pi_3(\Lambda)$.
\end{theorem}

A useful tool to construct 2-representation-finite algebras which are
derived equivalent to a given one is 2-APR-tilting, which is a higher
analog of usual APR-tilting.  The notion of 2-APR-co-tilting is
defined dually.

\begin{definition}
  \label{def:2-apr-tilting}
  \cite[Def. 3.14]{iyama_n-representation-finite_2011} Let $\Lambda$
  be a finite dimensional algebra of global dimension at most 2 and
  $\Lambda = P\oplus Q$ any direct summand decomposition of $\Lambda$
  such that
  \begin{enumerate}
  \item $\Hom_\Lambda(Q,P)=0$.
  \item $\Ext^i_\Lambda(\nu Q,P)=0$ for any $0<i\neq 2$.
  \end{enumerate}
  We call the complex
  \[
  T:=(\nu_2^{-1} P)\oplus Q\in\Db(\mod\Lambda)
  \]
  the \emph{2-APR-tilting complex associated with $P$}.
\end{definition}

In analogy with APR-tilting for hereditary algebras, 2-APR-tilting
preserves 2-representation-finiteness.

\begin{theorem}
  \label{thm:2-APR-2-RF}
  \cite[Thm. 4.7]{iyama_n-representation-finite_2011} Let $\Lambda$ be
  a 2-representation-finite algebra and $T$ a 2-APR-tilting complex in
  $\Db(\mod\Lambda)$.  Then the algebra $\End_{\Db(\Lambda)}(T)$ is
  also 2-representation-finite.
\end{theorem}

We can describe the effect of 2-APR-tilting using Theorem
\ref{thm:tildeLambda} as follows:

\begin{theorem}
  \label{thm:2-APR-QPs}
  \cite[Sec. 3.3]{iyama_n-representation-finite_2011} Let $\Lambda$ be
  a 2-representation-finite algebra and $P$ an indecomposable
  projective $\Lambda$-module which corresponds to a sink $k$ in the
  Gabriel quiver of $\Lambda$ and let $T$ be the associated
  2-APR-tilting $\Lambda$-module.  Also, let $(\Q,W,d)$ be the graded
  quiver with potential associated to $\Pi_3(\Lambda)$, see Theorem
  \ref{thm:tildeLambda}.  Then there is an isomorphism of graded
  algebras
  \[
  \End_\Lambda(T)\cong\Jac(\Q,W,d')
  \]
  where $d'$ coincides with $d$ on arrows not incident to $k$, for an
  arrow $a\in\Q$ incident to $k$ we have $d'(a)=1$ if $d(a)=0$, and we
  have $d'(a)=0$ if $d(a)=1$.
\end{theorem}

\subsection{The cluster category of $\cohX$}
\label{sec:cluster-category}

Cluster categories associated with hereditary algebras were introduced
in \cite{buan_tilting_2006} in order to categorify the combinatorics
of acyclic cluster algebras.  The cluster category of a weighted
projective line was studied in \cite{barot_cluster_2010},
\cite{barot_tubular_2012} and \cite{barot_tubular_2013}.  For the
point of view of this article, they arise as the categorical
environment of 3-preprojective algebras of endomorphism algebras of
tilting sheaves in $\cohX$.

The cluster category associated with $\cohX$ is by definition the
orbit category
\[
\C = \C_\XX := \Db(\cohX)/(\tau[-1]).
\]
Thus, the objects of $\C$ are bounded complexes of coherent sheaves
over $\XX$ and the morphism spaces are given by
\[
\Hom_\C(X,Y) := \bigoplus_{i\in\ZZ}\Hom_{\Db(\XX)}(X,\tau^iY[-i])
\]
with the obvious composition rule.  Note that $\Hom_\C(X,Y)$ has a
natural $\ZZ$-grading.  It is known \cite{barot_cluster_2010} that
$\C$ is a $\Hom$-finite, Krull-Schmidt, $K$-linear triangulated
category with the 2-Calabi-Yau property: There is a natural
isomorphism
\[
D\Hom_\C(X,Y)\cong\Hom_\C(Y,X[2])
\]
for every $X,Y$ in $\C$.  It follows from
\cite[Prop. 2.1]{barot_cluster_2010} that $\cohX$ is a complete system
of representatives of isomorphism classes in $\C$ and that we have a
natural isomorphism
\[
\Hom_\C(X,Y) \cong\Hom_\XX(X,Y) \oplus\Ext_\XX^1(X,\tau^{-1}Y)
\]
for $X,Y\in\cohX$.  Recall that an object $T$ in $\C$ is said to be
\emph{rigid} provided that $\Hom_\C(T,T[1])=0$; more strongly, if we
have that
\[
\add T = \setP{X\in\C}{\Hom_\C(X,T[1])=0},
\]
then $T$ is called a \emph{cluster-tilting object}.  Identifying
isomorphism classes in $\cohX$ with those in $\C_\XX$, it follows that
tilting (resp. rigid) sheaves in $\cohX$ are precisely cluster-tilting
(resp. rigid) objects in $\C$.  Moreover, we have the following
description of the endomorphism algebras of cluster-tilting objects in
$\C$.

\begin{proposition}
  \label{prop:EndC-Pi3}
  \cite[Prop. 4.7]{amiot_cluster_2009} Let $T$ be a tilting sheaf in
  $\cohX$.  Then there is an isomorphism of graded algebras between
  $\End_\C(T)$ and $\Pi_3(\End_\XX(T))$.
\end{proposition}

The category $\C$ has a cluster structure in the sense of
\cite{buan_tilting_2006}.  Moreover, mutation of cluster-tilting
objects is compatible with mutations of tilting sheaves and mutation
of Jacobian algebras, see \cite[Secs. 1, 2.5]{geiss_tubular_2013} for
example.

Finally, we have the following characterization of cluster-tilting
objects with selfinjective endomorphism algebra.

\begin{proposition}
  \label{prop:selfinj-QP-Sigma2}
  \cite[Prop. 4.4]{herschend_selfinjective_2011} Let $T$ be a
  cluster-tilting object in $\C$. Then $T\cong T[2]$ if and only if
  $\End_\C(T)$ is a selfinjective algebra.
\end{proposition}

\section{Proofs of the main results}

In this section we give the proofs of the main results of this
article, see Theorems \ref{thm:2-rf-der-can}, \ref{thm:t2-cpx} and
\ref{thm:classification-ct-can}.

Note that by the definition of the cluster category associated to
$\XX$, we have a commutative diagram of functors
\[
\begin{tikzcd}
  \cohX \rar{\tau}\dar & \cohX\dar\\
  \C_\XX \rar{\tau} & \C_\XX
\end{tikzcd}
\]
where the vertical arrows correspond to the canonical projection
functor. The following characterization can be easily deduced from
known results.

\begin{proposition}
  \label{prop:tau2-stable}
  Let $T$ be a tilting complex in $\Db(\cohX)$.  Then, the following
  conditions are equivalent:
  \begin{enumerate}
  \item\label{it:2H-2RF} The algebra $\End_\XX(T)$ is a
    2-representation-finite algebra.
  \item\label{it:selfinjective-QP} The algebra $\End_\C(T)$ is a
    finite-dimensional selfinjective algebra.
  \item\label{it:Sigma2-eq} We have $T[2] \cong T$ in $\C_\XX$ and
    $\End_{\Db(\XX)}(T)$ has global dimension at most 2.
  \end{enumerate}
  Moreover, if $T\in\cohX$, then any of the three equivalent
  conditions above is equivalent to $T$ being $\tau^2$-stable.
\end{proposition}
\begin{proof}
  \eqref{it:2H-2RF} is equivalent to \eqref{it:selfinjective-QP}. Let
  $\Lambda:=\End_\XX(T)$. By Proposition
  \ref{prop:charactetrization-2RF-selfinjective-QP}, the algebra
  $\Lambda$ is 2-representation-finite if and only if $\Pi_3(\Lambda
  )$ is a selfinjective finite dimensional algebra. Moreover,
  Proposition \ref{prop:EndC-Pi3} yields an isomorphism between
  $\Pi_3(\Lambda)$ and $\End_\C(T)$. The claim follows.  The
  equivalence between \eqref{it:selfinjective-QP} and
  \eqref{it:Sigma2-eq} is shown in Proposition
  \ref{prop:selfinj-QP-Sigma2}.

  Finally, let $T\in\cohX$. We show that \eqref{it:Sigma2-eq} is
  equivalent to $T$ being $\tau^2$-stable. Note that, by the
  definition of $\C$, the functors $[1]\colon\C\to\C$ and
  $\tau\colon\C\to\C$ are naturally isomorphic. Hence, we have
  $T[2]\cong\tau^2T$ as objects of $\C$ and, since isomorphism classes
  in $\C$ and $\cohX$ coincide, we have $T[2]\cong\tau^2T$ in
  $\cohX$. The claim follows.
\end{proof}

In the case of $\tau^2$-stable tilting sheaves we obtain further
restrictions on their endomorphism algebras.

\begin{proposition}
  \label{prop:tau2-2RF}
  Let $T$ be a tilting complex in $\Db(\cohX)$. Then, $T$ is
  $\tau^2$-stable if and only if $\End_{\Db(\XX)}(T)$ is a
  2-homogeneous 2-representation-finite algebra.
\end{proposition}
\begin{proof}
  Let $T$ be a tilting complex in $\Db(\cohX)$ and set
  $\Lambda:=\End_{\Db(\XX)}(T)$. Then, we have $\tau^2 T\cong T$ if
  and only if $(\nu[-1])^2(\Lambda)\cong \Lambda$ which is equivalent
  to $D\Lambda=\nu\Lambda \cong \nu_2^{-1}\Lambda$. Hence, to show
  that $\Lambda$ is a 2-homogeneous 2-representation-finite algebra,
  see Definition \ref{def:2RF}, we only need to show that if $T$ is
  $\tau^2$-stable then $\Lambda$ has global dimension at most
  2. Indeed, for each $i\geq 3$ we have
  \[
  \Ext_\Lambda^i(D\Lambda,\Lambda)\cong\Hom_{\Db(\Lambda)}(\nu^{-1}\Lambda[2],\Lambda[i])\cong
  D\Hom_{\Db(\Lambda)}(\Lambda[i-2],\Lambda)=0.
  \]
  Thus $\Lambda$ has global dimension at most 2 as required.
\end{proof}

The following result is crucial in our approach, as it allows to pass
from $\tau^2$-stable tilting complex to $\tau^2$-stable tilting
sheaves using 2-APR-(co)tilting.  Recall that the effect of
2-APR-(co)tilting on the endomorphism algebras of basic tilting
complexes can be described using mutations of graded quivers with
potential, see Theorem \ref{thm:2-APR-QPs}.

\begin{proposition}
  \label{prop:reduction-to-cohX}
  Let $T$ be a basic tilting complex in $\Db(\cohX)$ such that
  $\End_{\Db(\XX)}(T)$ is a 2-representation-finite algebra. Then,
  there exists a $\tau^2$-stable tilting sheaf $E\in\coh \XX$ obtained
  by iterated 2-APR-tilting from $T$.
\end{proposition}
\begin{proof}
  Since shifting does not change endomorphism algebras, we can assume
  that $T$ is concentrated in degrees $-\ell,\dots,-1,0$. Since
  $\cohX$ is hereditary, we have $T\cong
  T_\ell[-\ell]\oplus\dots\oplus T_1[-1]\oplus T_0$ where each $T_i$
  is a non-zero sheaf.  We proceed by induction on $\ell$. The case
  $\ell=0$ follows immediately from Proposition
  \ref{prop:tau2-stable}, so let $\ell>0$. We claim that the complex
  \[
  T': = (\tau^{-1} T_\ell)[1-\ell]\oplus
  T_{\ell-1}[1-\ell]\oplus\dots\oplus T_1[-1]\oplus T_0
  \]
  is a 2-APR-tilting complex.  Indeed, since there are no negative
  extensions between objects of $\cohX$, we have
  \[
  \bigoplus_{i=0}^{\ell-1}
  \Hom_{\Db(\XX)}(T_i[-i],T_\ell[-\ell])=\bigoplus_{i=0}^{\ell-1}
  \Hom_{\Db(\XX)}(T_i,T_\ell[i-\ell])=0.
  \]
  
  Moreover, using the identity $\nu=\tau[1]$, we obtain
  \begin{align*}
    \bigoplus_{i=0}^{\ell-1} \Ext_{\Db(\XX)}^1(\nu
    T_i[-i],T_\ell[-\ell])
    \cong& \bigoplus_{i=0}^{\ell-1}\Hom_{\Db(\XX)}(\tau T_i[-i],T_\ell[-\ell])\\
    \cong& \bigoplus_{i=0}^{\ell-1} \Hom_{\Db(\XX)}(\tau
    T_i,T_\ell[i-\ell])=0.
  \end{align*}
  Finally, since $\End_{\Db(\XX)}(T)$ has global dimension 2 we have
  that
  \[
  \bigoplus_{i=0}^{\ell-1} \Ext_{\Db(\XX)}^j(\nu
  T_i[-i],T_\ell[-\ell])=0
  \]
  for all $j\geq 3$. This shows that $T'$ is a 2-APR tilting complex
  and, by Theorem \ref{thm:2-APR-2-RF}, we have that
  $\End_{\Db(\XX)}(T')$ is a 2-representation-finite algebra. Hence,
  by the induction hypothesis, by iterated 2-APR-tilting we can
  construct a $\tau^2$-stable tilting sheaf $E$ from $T$.
\end{proof}

Next, we determine which weighted projective lines can have
$\tau^2$-stable tilting sheaves.

\begin{proposition}
  \label{prop:cases}
  Let $T\in\coh\XX$ be a $\tau^2$-stable tilting sheaf.  Then $\XX$
  has tubular weight type $(2,2,2,2)$, $(2,4,4)$ or $(2,3,6)$.
\end{proposition}
\begin{proof}
  Since there are no tilting sheaves of finite length, we have
  $\slope(T)\in\QQ$.  Moreover, as we have $\tau^2 T\cong T$, it
  follows from Proposition \ref{prop:slope} that $\delta(\vecw)=0$.
  Then, using equation \eqref{eq:delta-euler}, we have that
  $\chi(\XX)=0$ hence $\XX$ has tubular type.

  We recall if $\XX$ has tubular type, then the full subcategory of
  $\cohX$ given of all sheaves of a fixed slope is equivalent to the
  category $\coh_0\XX$ of torsion sheaves over $\XX$
  \cite[Thm. 3.10]{lenzing_weighted_2011}.  Assume now that $X$ is an
  indecomposable summand of $T$ which belongs to a tube of odd period
  $2a+1$, so we have $X \cong \tau(\tau^{2a} X)$.  By hypothesis,
  $\tau^{2a} X$ is a direct summand of $T$. Hence, by Serre duality we
  have
  \[
  0 = \Ext_\XX^1(\tau^{2a}X,X) \cong
  D\Hom_\XX(X,\tau(\tau^{2a}X))=D\Hom_\XX(X,X),
  \]
  a contradiction.  Hence every indecomposable summand of $T$ belongs
  to a tube of even period. This rules out weight type $(3,3,3)$.
  Therefore must have tubular weight type $(2,2,2,2)$, $(2,4,4)$ or
  $(2,3,6)$.
\end{proof}

The following result gives a classification of the endomorphism
algebras of basic $\tau^2$-stable tilting sheaves in $\cohX$.

\begin{theorem}[Lenzing]
  \label{thm:classification-tau2X}
  Let $T$ be a basic $\tau^2$-stable tilting sheaf in $\cohX$.  Then
  $\End_\XX(T)$ is isomorphic to one of the algebras indicated in
  Figures \ref{fig:2222-2-APR}, \ref{fig:244-multi} or
  \ref{fig:236-multi}.
  \begin{figure}[t]
    \includegraphics[scale=1]{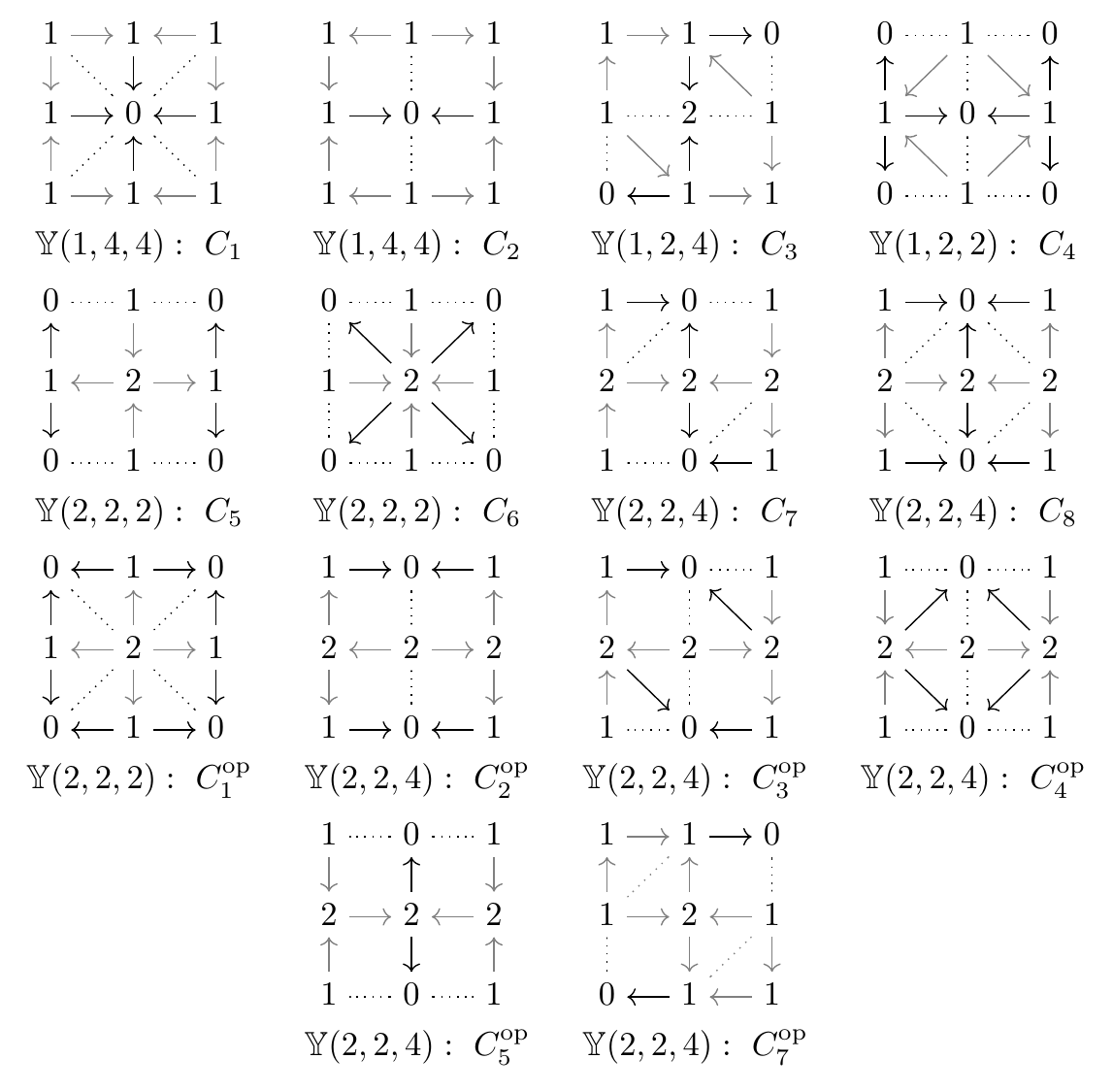}
    \caption{Endomorphism algebras of $\tau^2$-stable tilting sheaves
      in normal position over a weighted projective line of weight
      type $(2,4,4)$. We have indicated $\End_\XX(T_+)$ by gray
      arrows. Note that $\tau^2$ acts on each configuration by
      rotation by $\pi$ and that $C_6$ and $C_8$ are
      self-opposite. The weight type of the reduced weighted
      projective line $\YY$ is indicated for reference.}
    \label{fig:244-multi}
  \end{figure}
  \begin{figure}[t]
    \includegraphics[scale=1]{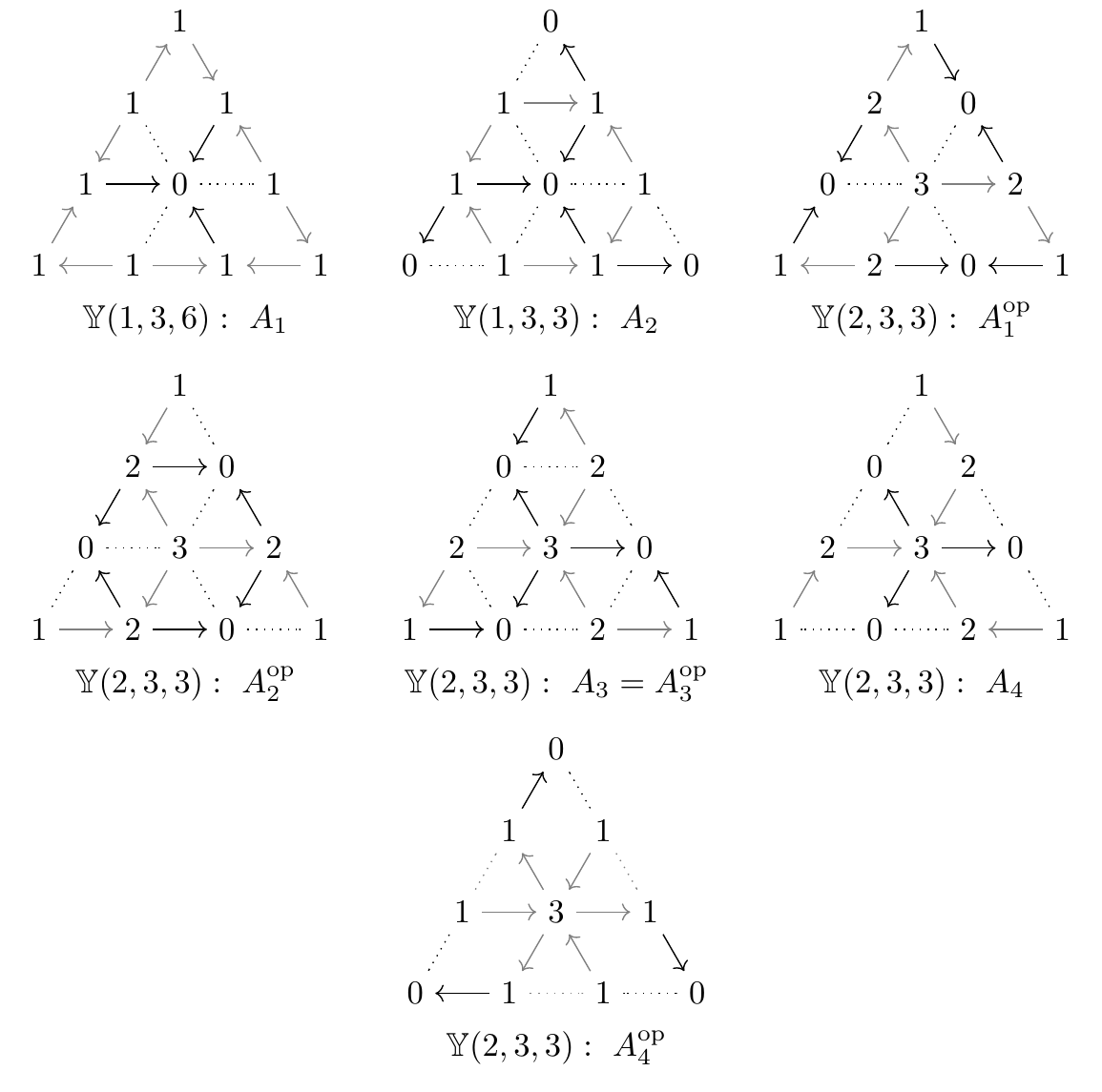}
    \caption{Endomorphism algebras of $\tau^2$-stable tilting sheaves
      in normal position over a weighted projective line $\XX$ of
      weight type $(2,3,6)$.  We have indicated $\End_\XX(T_+)$ by
      gray arrows. Note that $\tau^2$ acts on each configuration by
      left rotation by $\pi/3$. The weight type of the reduced
      weighted projective line $\YY$ is indicated for reference.}
    \label{fig:236-multi}
  \end{figure}
\end{theorem}
\begin{proof}
  First, suppose that $\XX$ has type $(2,2,2,2;\lambda)$.  Since
  $\tau^2$ is the identity in $\cohX$, all tilting sheaves are
  $\tau^2$-stable in this case.  Their endomorphism algebras are
  known, see Skowro{\'n}ski
  \cite[Ex. 3.3]{skowronski_selfinjective_1989} (see also Figure
  \ref{fig:2222-2-APR}).
  
  For the other cases, weight types $(2,4,4)$ and $(2,3,6)$, we rely
  on the following argument which is due to Lenzing.  Let $T$ be a $\tau^2$-stable tilting
  sheaf in $\cohX$.  By Proposition \ref{prop:normal-position}, we can
  assume that $T$ is in normal position. By Lemma
  \ref{lemma:T0-semisimple} every indecomposable direct summand of
  $T_0$ is an exceptional simple sheaf.  Also, the perpendicular
  category $T_0^\perp$ is equivalent to a category of the form
  $\coh\YY$ where $\YY$ is a weighted projective line with
  $\chi(\YY)>0$. Moreover, there exist a finite dimensional algebra
  $H$ of extended Dynkin type $\Delta$ such that $\coh\YY$ is derived
  equivalent to $\coh\YY$. In addition, we have $T_+\in\vect\YY$, see
  Proposition \ref{thm:techincal-general}. It follows that
  $\End_\XX(T_+)=\End_\YY(T_+)$ is isomorphic to the endomorphism
  algebra of a preprojective $H$-module.  Note that $\End_\YY(T_+)$
  must admit an action of order $p/2$ which does not fix any line
  bundle summands of $T_+$, see by Proposition
  \ref{thm:techincal-general}\eqref{it:action-vecY}. Finally,
  Proposition \ref{thm:multiextension} yields an isomorphism of
  algebras
  \[
  \End_\XX(T)\cong\End_\YY(T_+\oplus E)\cong
  \begin{bmatrix}
    \End_\YY(T_+)&\End_\YY(T_+,E)\\
    0 & \End_\YY(E)
  \end{bmatrix}
  \]
  where $E$ is the direct sum of all regular simple modules in
  $\coh\YY$ in the exceptional tubes concentrated in the $\lambda_i$'s
  such that $q_i\neq 0$ (note that here $q_i$ must be either zero or
  $p_i/2$). It follows that $\End_\YY(E)$ is a semisimple algebra with
  $q_1+\cdots+q_t$ simple modules.

  Hence, to prove the theorem we only need to do the following:
  \begin{enumerate}
  \item Take a vector bundle in $T_+\in\coh\YY$ whose endomorphism
    algebra admits a symmetry of order 2 for $\XX$ of type (2,4,4) or
    order 3 for $\XX$ of type (2,3,6) not fixing any line bundles.
  \item Compute the algebra $\End_\XX(T_+\oplus E)$.
  \item Check if $\End_\XX(T_+\oplus E)$ is a 2-representation-finite
    algebra, see Proposition \ref{prop:tau2-stable}.
  \end{enumerate}
  This process, although lengthy, is straightforward. We illustrate
  part of it for $\XX$ of weight type $(2,4,4)$. The case were $\XX$
  has type $(2,3,6)$ is completely analogous. The cases we need to
  deal with are stated in Table \ref{table:cases}.

  \begin{table}
    \begin{center}
      \begin{tabular}{cc}
        \hline\noalign{\smallskip}
        \multicolumn{2}{c}{$\XX(2,4,4)$} \\
        \noalign{\smallskip}\hline\noalign{\smallskip}
        $\YY$ & $\Delta$ \\
        \noalign{\smallskip}\hline\noalign{\smallskip}
        (1,4,4) & $\tilde{A}_{4,4}$ \\
        (1,2,4) & $\tilde{A}_{2,4}$ \\
        (1,2,2) & $\tilde{A}_{2,2}$\\
        (2,2,2) & $\tilde{D}_4$ \\    
        (2,2,4) & $\tilde{D}_6$ \\
        \noalign{\smallskip}\hline
      \end{tabular}\qquad
      \begin{tabular}{ccccc}
        \hline\noalign{\smallskip}
        \multicolumn{2}{c}{$\XX(2,3,6)$} \\
        \noalign{\smallskip}\hline\noalign{\smallskip}
        $\YY$ & $\Delta$ \\
        \noalign{\smallskip}\hline\noalign{\smallskip}
        (1,3,6) & $\tilde{A}(3,6)$ \\
        (1,3,3) & $\tilde{A}(3,3)$ \\
        (2,3,3) & $\tilde{E}_6$ \\
        & \phantom{$\tilde{D}_4$} \\
        & \phantom{$\tilde{D}_4$} \\
        \noalign{\smallskip}\hline
      \end{tabular}
    \caption{Possible weight types for $\YY$ and the extended Dynkin
      type $\Delta$ of the associated hereditary algebra.}
    \label{table:cases}
    \end{center}
  \end{table}

  $\YY(1,4,4)$ In this case we have $\Delta=\tilde{A}_{4,4}$. The only
  possibility for the Gabriel quiver of $\End_\YY(T_+)$ is a
  non-oriented cycle with 8 vertices. Moreover, it must have 4 arrows
  pointing in clockwise direction and 4 arrows pointing in
  counterclockwise direction.  These are the quivers highlighted in
  the algebras $C_1$ and $C_2$ Figure \ref{fig:244-multi}. All of
  these algebras are 2-representation-finite algebras, as can be
  readily verified by checking that their 3-preprojective algebras are
  selfinjective, see Proposition
  \ref{prop:charactetrization-2RF-selfinjective-QP}. The reader can
  verify that they indeed arise by the procedure described in Theorem
  \ref{thm:multiextension}.

  The cases $\YY(1,2,4)$ and $\YY(1,2,2)$ are completely analogous,
  The resulting 2-representation finite algebras correspond to $C_3$
  and $C_4$ respectively in Figure \ref{fig:244-multi}.
  
  $\YY(2,2,2)$ In this case we have $\Delta=\tilde{D}_4$.  The only
  possible endomorphism algebras of preprojective tilting $H$-modules
  are orientations of the Dynkin diagram of type $\tilde{D}_6$
  \[
  \begin{tikzcd}[column sep=small, row sep=small]
    & 1\dar[path] \\
    1\rar[path] & 2\rar[path]\dar[path] & 1\\
    & 1
  \end{tikzcd}
  \]
  or the canonical algebra of type $(2,2,2)$, see Happel-Vossieck's
  list \cite{happel_minimal_1983}.  The only quivers which admit an
  action of order 2 which does not fixes any line bundle summand of
  $T_+$ are the ones highlighted in algebras $C_1^\op$, $C_5$ and
  $C_6$ in Figure \ref{fig:244-multi}, corresponding to symmetric
  orientations of the Dynkin diagram above.

  $\YY(2,2,4)$ We have $\Delta=\tilde{D}_6$. In this case (and only in
  this case), there are algebras in Happel-Vossieck's list which have
  an action of order $p/2$ which do not extend to a
  2-representation-finite algebra. One way to rule out these algebras
  before doing any computation is to determine the action induced by
  $\tau^2$ on the Auslander-Reiten quiver of $\vect\YY$, which has
  shape $\ZZ\tilde{D}_6$, see Theorem
  \ref{thm:techincal-general}\eqref{it:Y-H} and \cite[Table
  1]{lenzing_weighted_2011}. We prove below that this
  action is given by rotation along the horizontal axis of $\vect\YY$,
  corresponding to the action given by degree shift by
  $\vecy+2\vecw_\YY\in\LL(2,2,4)$. Taking this into account, according
  to \cite{happel_minimal_1983} the possible endomorphism algebras of
  preprojective tilting $H$-modules are given in Figure
  \ref{fig:HV-D6}.
  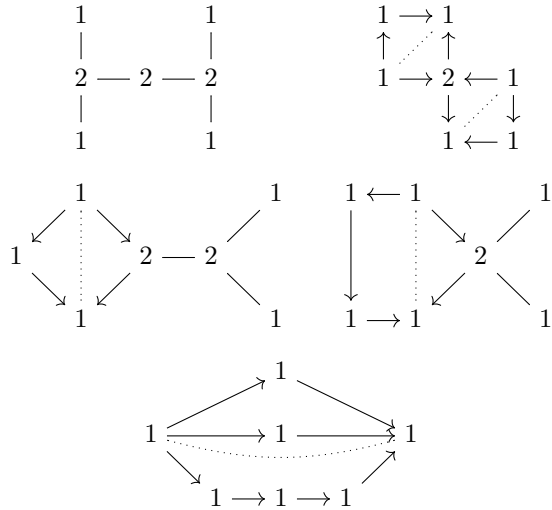
\begin{figure}
    \begin{center}
      \begin{tabular}{cc}
        \begin{tikzcd}[ampersand replacement=\&, column sep=small, row
          sep=small]
          1\dar[path] \& \& 1\dar[path]\\
          2 \rar[path]\dar[path] \& 2\rar[path] \& 2\dar[path]\\
          1 \& \& 1
        \end{tikzcd}
        & 
        \begin{tikzcd}[ampersand replacement=\&, column sep=small, row
          sep=small]
          1 \rar \& 1\dlar[dotted,path]\\
          1\uar\rar \& 2\dar\uar \& 1\lar\dar\dlar[dotted,path]\\
          \& 1 \& 1\lar
        \end{tikzcd} \\
        \begin{tikzcd}[ampersand replacement=\&, column sep=small, row
          sep=small]
          \&1\dlar\drar\ar[dotted,path]{dd}\&\&\& 1\\
          1\drar\&\&2\rar[path]\dlar\&2\urar[path]\drar[path]\\
          \&1\&\&\&1
        \end{tikzcd}
        & \begin{tikzcd}[ampersand replacement=\&, column sep=small, row
          sep=small]
          1 \ar{dd} \& 1 \lar\drar\ar[dotted,path]{dd} \& \& 1\\
          \& \&2 \dlar\urar[path]\drar[path]\\
          1\rar \& 1 \& \& 1        
        \end{tikzcd} \\
        \multicolumn{2}{c}{
          \begin{tikzcd}[ampersand replacement=\&, column sep=small, row
            sep=small]
            \&\&1\ar{rrd}\\
            1\ar{rru}\drar\ar{rr}\ar[out=-15,in=-165,path,dotted]{rrrr}\&\&1\ar{rr}\&\&1\\
            \&1\rar\&1\rar\&1\urar
          \end{tikzcd}}
      \end{tabular}
    \end{center}
    \caption{Endomorphism algebras of preprojective tilting modules of
      type $\tilde{D}_6$ with dimension vectors. The orientation of
      simple edges can be chosen arbitrarily and the relations
      indicate that the sum of all paths with the corresponding
      endpoints is zero.}
    \label{fig:HV-D6}
  \end{figure}
  The only quivers in Figure \ref{fig:HV-D6} which are stable under
  rotation by $\pi$ along the horizontal axis of the Auslander-Reiten
  quiver of $\vect\YY$ are the ones highlighted in algebras $C_2^\op$,
  $C_3^\op$, $C_4^\op$, $C_5^\op$, $C_7$, $C_7^\op$ and $C_8$ in
  Figure \ref{fig:244-multi}.

  Finally, let us prove that the action induced by $\tau^2$ on
  $\vect\YY$ is indeed given by rotation along the horizontal axis.
  This also serves as an example of the method to
  compute $\End_\XX(T)$ using Theorem \ref{thm:multiextension}.


  Let $\XX$ be a weighted projective line of tubular type $(2,4,4)$.
  Let $X$ be an exceptional simple sheaf concentrated at $\lambda_2$
  and set $T_0:=X\oplus\tau^2X$.  We write $\LL(2,2,4)=\langle\vecx,\vecy,\vecz,\vecc\mid
  2\vecx=2\vecy=4\vecz=\vecc\rangle$ and $\vecw=\vecw_\YY$.  Also, we
  put $R:=R(2,2,4)$

  Let $T_+\in\vect\YY$ be the tilting bundle indicated in Figure
  \ref{fig:multiextension-ex} and $E=S\oplus S'$ be the direct sum of
  the two exceptional simple sheaves in $\coh\YY$ concentrated at the
  point $\lambda_2$.
  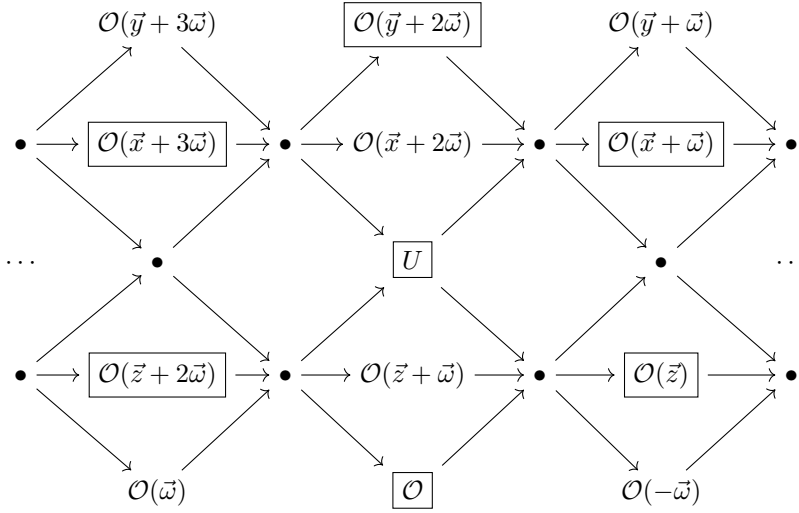
\begin{figure}
    \begin{center}
      \begin{tikzcd}[column sep=small, row sep=normal]
        {}&\OO(\vecy+3\vecw)\drar&&\tikzbox{\OO(\vecy+2\vecw)}\drar&&\OO(\vecy+\vecw)\drar&\\
        \bullet\urar\rar\drar&\tikzbox{\OO(\vecx+3\vecw)}\rar&\bullet\drar\urar\rar&\OO(\vecx+2\vecw)\rar&\bullet\drar\rar\urar&\tikzbox{\OO(\vecx+\vecw)}\rar&\bullet\\
        \cdots&\bullet\urar\drar&&\tikzbox{U}\drar\urar&&\bullet\drar\urar&\cdots\\
        \bullet\urar\drar\rar&\tikzbox{\OO(\vecz+2\vecw)}\rar&\bullet\urar\drar\rar&\OO(\vecz+\vecw)\rar&\bullet\rar\urar\drar&\tikzbox{\OO(\vecz)}\rar&\bullet\\
        &\OO(\vecw)\urar&&\tikzbox{\OO}\urar&&\OO(-\vecw)\urar
      \end{tikzcd}
    \end{center}
    \caption{Tilting bundle in $\vect\YY$. Black vertices indicate vector bundles
      of rank 2.}
    \label{fig:multiextension-ex}
  \end{figure}
  Put $T:=T_+\oplus T_0$. By Theorem \ref{thm:multiextension} we have
  an isomorphism of $K$-algebras
  \[
  \End_\XX(T)\cong
  \begin{bmatrix}
    \End_\YY(T_+) & \Hom_\YY(T_+,E) \\
    0 & \End_\YY(E)\cong K\times K
  \end{bmatrix}.
  \]
  We need to compute $\End_\YY(T_+,E)$.  For this, recall from Theorem
  \ref{thm:props-cohX}\eqref{it:simples} that we have short exact
  sequences
  \begin{equation}
    \label{eq:X}
    \begin{tikzcd}
      0\rar&\OO_\YY(-\vecy)\rar&\OO_\YY\rar&S\rar&0
    \end{tikzcd}
  \end{equation}
  and
  \begin{equation}
    \label{eq:Y}
    \begin{tikzcd}
      0\rar&\OO_\YY(-2\vecy)\rar&\OO_\YY(-\vecy)\rar&S'\rar&0.
    \end{tikzcd}
  \end{equation}
  We shall compute $\dim\Hom_\YY(\OO(\vecz),S)$ and
  $\dim\Hom_\YY(\OO(\vecz),S')$ by applying the functor
  $\Hom_\YY(\OO(\vecz),-)$. Before that, it its convenient to make
  some preliminary calculations.

  Firstly, by Theorem \ref{thm:props-cohX}\eqref{it:morphisms} we have
  $\Hom_\YY(\OO(\vecz),\OO)=0$ and using Serre duality we obtain
  \[
  \Ext_\YY^1(\OO(\vecz),E)\cong D\Hom_\XX(E,\OO(\vecz+\vecw))=0,
  \]
  since there are no non-zero morphisms from a torsion sheaf to a
  vector bundle. Secondly, again by Theorem
  \ref{thm:props-cohX}\eqref{it:morphisms} and Serre duality we have
  \[
  \Ext_\YY^1(\OO(\vecz),\OO)\cong D\Hom_\YY(\OO,\OO(\vecz+\vecw))\cong
  R_{\vecx-\vecy}=0.
  \]
  Simlarly, we have
  \[
  \Ext_\YY^1(\OO(\vecz),\OO(-\vecy))\cong
  D\Hom_\YY(\OO(-\vecy),\OO(\vecz+\vecw))=R_{\vecw+\vecy+\vecz}=R_{\vecx}.
  \]
  In addition, we have
  \[
  \Hom_\YY(\OO(\vecz),\OO)=R_{-\vecz}=0\quad\Hom_\YY(\OO(\vecz),\OO(-\vecy))=R_{-\vecy-\vecz}=0.
  \]
  Hence, applying the functor $\Hom_\YY(\OO_\YY(\vecz),-)$ to the
  sequences \eqref{eq:X} and \eqref{eq:Y} yields exact sequences
  \[
  \begin{tikzcd}[column sep=small]
    0\rar&\Hom_\YY(\OO(\vecz),S)\rar&\Ext_\YY^1(\OO(\vecz),\OO(-\vecy))\rar&0
  \end{tikzcd}
  \]
  and
  \[
  \begin{tikzcd}[column sep=small]
    0\rar&\Hom_\YY(\OO(\vecz),S')\rar&\Ext_\YY^1(\OO(\vecz),\OO(-2\vecy))\rar&\Ext_\YY^1(\OO(\vecz),\OO(-\vecy))\rar&0
  \end{tikzcd}
  \]
  Then, by Theorem \ref{thm:props-cohX}\eqref{it:morphisms} we have
  \begin{align*}
    \dim\Hom_\YY(\OO(\vecz),S)=&\dim\Ext_\YY^1(\OO(\vecz),\OO(-\vecy))\\
    =&\dim\Hom_\YY(\OO(-\vecy),\OO(\vecz+\vecw))\\
    =&\dim R_{\vecw+\vecz+\vecy}\\
    =&\dim R_{\vecx}=1
  \end{align*}
  and
  \begin{align*}
    \dim\Hom_\YY(\OO(\vecz),S')=&\dim
    \Ext_\YY^1(\OO(\vecz),\OO(-2\vecy))-\dim\Ext_\YY^1(\OO(\vecz),\OO(-\vecy))\\
    =&\dim\Hom_\YY(\OO(-2\vecy),\OO(\vecz+\vecw))-1\\
    =&\dim R_{2\vecy+\vecz+\vecw}-1\\
    =&\dim R_{\vecx+\vecy}-1=0.
  \end{align*}
  A similar argument shows that
  \[
  \Hom_\YY(\OO(\vecx+\vecw),S)=0
  \]
  and
  \[
  \dim\Hom_\YY(\OO(\vecx+\vecw),S')=1.
  \]
  Proceeding in the same fashion, the reader can verify the equalities
  \begin{eqnarray*}
    &\dim\Hom_\YY(\OO(\vecx+3\vecw),S)=0,\quad\dim\Hom_\YY(\OO(\vecx+3\vecw),S')=1,\\
    &\dim\Hom_\YY(\OO(\vecy+2\vecw),S)=0,\quad\dim\Hom_\YY(\OO(\vecy+2\vecw),S')=1,\\    
    &\dim\Hom_\YY(\OO(\vecz+2\vecw),S)=1,\quad\dim\Hom_\YY(\OO(\vecz+2\vecw),S')=0,\\    
    &\dim\Hom_\YY(\OO,S)=1,\quad\dim\Hom_\YY(\OO,S')=0.
  \end{eqnarray*}
  
  It remains to compute $\dim\Hom_\YY(U,S)$ and
  $\dim\Hom_\YY(U,S')$. We have an exact sequence
  \[
  \begin{tikzcd}
    0\rar&\OO(\vecw)\rar&U\rar&\OO(\vecz)\rar&0.
  \end{tikzcd}
  \]
  Applying the contravariant functor $\Hom_\YY(-,S)$ yields an exact
  sequence
  \[
  \begin{tikzcd}[column sep=tiny]
    0\rar&\Hom_\YY(\OO(\vecz),S)\rar&\Hom_\YY(U,S)\rar&\Hom_\YY(\OO(\vecw),S)\rar&\Ext_\YY^1(\OO(\vecz),S)=0.
  \end{tikzcd}
  \]
  We already know that $\dim\Hom_\YY(\OO(\vecz),S)=1$. Proceeding as
  before, applying the functor $\Hom_\YY(\OO(\vecw),-)$ to the short
  exact sequence \eqref{eq:X}, we can show that
  $\dim\Hom_\YY(\OO(\vecw),S)=0$. Hence $\dim\Hom_\YY(U,S)=1$. We can
  show that $\dim\Hom_\YY(U,S')=1$ in a similar manner.

  It follows, by a suitable change of basis of $\End_\YY(T_+)$, that
  the quiver with relations of $\End_\XX(T)^\op$ is given by
  \[
  \begin{tikzcd}[column sep=small]
    \OO\rar&\OO(\vecz)\rar\dlar[path,dotted]&S\dar[path,dotted]\\
    \OO(\vecz+2\vecw)\rar\uar&U\dar\uar&\OO(\vecx+3\vecw)\lar\dar\\
    S'\uar[path,dotted]&\OO(\vecx+\vecw)\lar\urar[path,dotted]&\OO(\vecy+2\vecw)\lar
  \end{tikzcd}
  \]
  where each relation is a zero relation or a commutative relation.

  Using Proposition \ref{prop:charactetrization-2RF-selfinjective-QP}
  it is easy to check that $\End_\XX(T)$ is a 2-representation-finite
  algebra.  Therefore $T$ is a $\tau^2$-stable tilting sheaf, see
  Proposition \ref{prop:tau2-stable}.  Then, we can see in Figure
  \ref{fig:multiextension-ex} that the only action of order 2 on the
  Auslander-Reiten quiver of $\vect\YY$ which fixes $T_+$ given by
  rotation by $\pi$ along the horizontal axis, which can be
  interpreted as degree shift by $\vecy+2\vecw_\YY$. This
  concludes the proof of the theorem.
\end{proof}

As a consequence of Theorem \ref{thm:classification-tau2X} we obtain
the following classification results.

\begin{theorem}
  \label{thm:2-rf-der-can}
  Let $\XX$ be a weighted projective line and basic $T$ a tilting complex in
  $\Db(\XX)=\Db(\cohX)$. Then $\End_{\Db(\XX)}(T)$, is a
  2-representation-finite algebras if and only if $\End_{\Db(\XX)}(T)$
  is one of the algebras in Figures \ref{fig:2222-2-APR},
  \ref{fig:244-2RF} and \ref{fig:236-2RF}.  Moreover, this determines
  $T$ up to an autoequivalence of $\Db(\cohX)$.
\end{theorem}
\begin{proof}
  The first claim follows immediately from Proposition
  \ref{prop:reduction-to-cohX} and Theorem
  \ref{thm:classification-tau2X}, since \ref{fig:244-2RF} and
  \ref{fig:236-2RF} are all algebras that can be obtained by
  2-APR-(co)tilting from the algebras in Figures \ref{fig:244-2-APR}
  and \ref{fig:236-2-APR}. The second claim is a standard application
  of \cite[Thm. 3.2]{lenzing_exceptional_2002}.
\end{proof}

\begin{theorem}
  \label{thm:t2-cpx}
  Let $T$ be a basic complex in $\Db(\cohX)$. Then, $T$ is
  $\tau^2$-stable if and only if $\End_{\Db(\XX)}(T)$ is one of the
  algebras in Figures \ref{fig:2222-2-APR}, \ref{fig:244-2-APR}, or
  \ref{fig:236-2-APR}. Moreover, this determines $T$ up to an
  autoequivalence of $\Db(\cohX)$.
\end{theorem}
\begin{proof}
  By Theorem \ref{thm:2-rf-der-can}, the algebra $\End_\XX(T)$ can be
  obtained from one of the algebras in Figures \ref{fig:2222-2-APR},
  \ref{fig:244-multi} or \ref{fig:236-multi} by iterated
  2-APR-(co)-tilting. By Proposition \ref{prop:tau2-stable}, we have
  that $T$ is $\tau^2$-stable is a 2-homogeneous
  2-representation-finite algebra. These are precisely the algebras in
  Figures \ref{fig:2222-2-APR}, \ref{fig:244-2-APR}, or
  \ref{fig:236-2-APR}.
\end{proof}

\begin{theorem}
  \label{thm:classification-ct-can}
  Let $T$ be a basic $\tau^2$-stable tilting sheaf in $\cohX$.  Then
  the cluster-tilted algebra $\End_\C(T)$ is isomorphic to the
  Jacobian algebra associated to one of the quivers with potential in
  Figures \ref{fig:2222}, \ref{fig:244} or \ref{fig:236}, and all of
  the Jacobian algebras associated to one of these quivers with
  potential arise in this way.
\end{theorem}
\begin{proof}
  Let $T$ be a basic $\tau^2$-stable tilting sheaf in $\cohX$ and set
  $\Lambda=\End_\XX(T)$.  It follows from Theorem
  \ref{thm:classification-tau2X} that $\Lambda$ is isomorphic to one
  of the algebras in Figures \ref{fig:2222-2-APR}, \ref{fig:244-multi}
  or \ref{fig:236-multi}.  Then, by Proposition \ref{prop:EndC-Pi3}
  there exist an isomorphism $\End_\C(T)\cong \Pi_3(\Lambda)$.  By
  Theorem \ref{thm:tildeLambda}, we have that $\Pi_3(\Lambda)$ is
  isomorphic to the Jacobian algebra to one of the quivers with
  potential in Figures \ref{fig:2222}, \ref{fig:244} or \ref{fig:236}.
  
  Conversely, each Jacobian algebra associated to one of the quivers
  with potential in Figures \ref{fig:2222}, \ref{fig:244} or
  \ref{fig:236} is of the form $\Pi_3(\Lambda)$ for some $\Lambda$ in
  Figures \ref{fig:2222-2-APR}, \ref{fig:244-multi} or
  \ref{fig:236-multi}, see \cite[Secs. 5.1, 9.2 and
  9.3]{herschend_selfinjective_2011}.  The theorem follows.
\end{proof}

\begin{acknowledgements}
The author wishes to thank Prof. H. Lenzing for sharing his notes on
the classification of $\tau^2$-stable tilting sheaves (Theorem
\ref{thm:classification-tau2X}), where this article found its origin.
This thanks are extended to L. Demonet for motivating discussions on
$G$-equivariant categories, which did not find a place in the final
version of this article.  Finally, the author wishes to acknowledge
Prof. O. Iyama who suggested the author to work on this problems and
for the interesting discussions regarding higher Auslander-Reiten theory.
\end{acknowledgements}

\bibliographystyle{abbrv} \bibliography{zotero}

\end{document}